\documentclass[11pt]{amsart}
\usepackage[top=1in,bottom=1in,left=1in,right=1in,heightrounded]{geometry}
\usepackage{setspace}
\usepackage[parfill]{parskip}
\allowdisplaybreaks
\usepackage{nicematrix}

\usepackage{amsmath,amsthm,amssymb} 
\usepackage{mathtools}
\usepackage{mathrsfs}
\usepackage{stmaryrd}

\usepackage{tikz}
\usepackage{tikz-cd}
\tikzset{>={Straight Barb}}
\usetikzlibrary{decorations.pathmorphing} 
\usepackage{pgfplots}
    \pgfplotsset{compat=1.18}
\usepackage{graphicx}
\usepackage{adjustbox, multicol}

\usepackage{enumitem}

\usepackage{caption} 
\usepackage{hyperref}
\hypersetup{colorlinks, citecolor=blue, filecolor=blue, linkcolor=blue, urlcolor=blue}
\usepackage[capitalize,noabbrev]{cleveref}

\usepackage[numbers]{natbib}
\usepackage{mathpazo,soul}
\usepackage[cal=boondox,scr=boondox]{mathalfa}

\AtBeginDocument{
  \DeclareSymbolFont{AMSb}{U}{msb}{m}{n}
  \DeclareSymbolFontAlphabet{\mathbb}{AMSb}}

\DeclareMathAlphabet{\mathbx}{U}{BOONDOX-ds}{m}{n}
\SetMathAlphabet{\mathbx}{bold}{U}{BOONDOX-ds}{b}{n}
\DeclareMathAlphabet{\mathbbx} {U}{BOONDOX-ds}{b}{n}

\SetMathAlphabet{\mathcal}{bold}{U}{dutchcal}{b}{n}
\DeclareMathAlphabet{\mathbcal}{U}{dutchcal}{b}{n}

\usepackage{fancyhdr}

\theoremstyle{plain}
\newtheorem{theorem}{Theorem}[section]
\newtheorem{lemma}[theorem]{Lemma}
\newtheorem{proposition}[theorem]{Proposition}
\newtheorem{corollary}[theorem]{Corollary}
\newtheorem{introthm}{Theorem}

\theoremstyle{definition}
\newtheorem{example}[theorem]{Example}
\newtheorem{definition}[theorem]{Definition}
\newtheorem{remark}[theorem]{Remark}

\newtheorem{notation}{Notation}

\newcommand{\bt}{\mathrm{BT}}

\newcommand{\ff}{\mathbx{F}}
\newcommand{\qq}{\mathbx{Q}}
\newcommand{\zz}{\mathbx{Z}}
\newcommand{\rr}{\mathbx{R}}
\newcommand{\cc}{\mathbx{C}}
\newcommand{\uxx}{\mathbx{X}}
\newcommand{\kk}{\mathbx k}

\newcommand{\linspan}{\operatorname{Span}}
\newcommand{\spec}{\operatorname{Spec}}

\renewcommand{\epsilon}{\varepsilon}
\newcommand{\llangle}{\langle\hspace*{-3pt}\langle}
\newcommand{\rrangle}{\rangle\hspace*{-3pt}\rangle}

\title{Analogues of Traverso's Isogeny Conjecture for Some Unitary \lowercase{\emph{p}}-Divisible Groups} 

\author{Emerald Andrews} 
\address{Emerald Andrews\newline Department of Mathematics and Computer Science, Washington College, Chestertown, MD 21620, USA}
\email{estacy2@washcoll.edu}

\author{Deewang Bhamidipati} 
\address{Deewang Bhamidipati\newline Department of Mathematics and Statistics, Carleton College, Northfield, MN 55057, USA} 
\email{bdeewang@carleton.edu}

\author{Maria Fox} 
\address{Maria Fox\newline Department of Mathematics, Oklahoma State University, Stillwater, OK 74078, USA}
\email{maria.fox@okstate.edu}

\author{Heidi Goodson} 
\address{Heidi Goodson\newline Department of Mathematics, Brooklyn College and The Graduate Center, City University of New York, Brooklyn, NY 11210 USA}
\email{heidi.goodson@brooklyn.cuny.edu}

\author{\\Steven R. Groen} 
\address{Steven R. Groen\newline Korteweg-de Vries Institute for Mathematics, University of Amsterdam, Amsterdam, Netherlands}
\email{s.r.groen@uva.nl}

\author{Sandra Nair} 
\address{Sandra Nair\newline Department of Mathematics, University of Arizona, Tucson, AZ 85721, USA}
\email{sdnair@arizona.edu}

\begin{document}

\begin{abstract}
Let $p$ be a prime and $\kk$ an algebraically closed field of characteristic $p$. Two invariants worthy of study for a $p$-divisible group $X$ over $\kk$ are its isogeny cutoff $b_X$, which measures the amount of $p$-torsion necessary to determine the Newton polygon of $X$, and its minimal height $q_X$, which measures its distance to the closest minimal $p$-divisible group (in the sense of Oort). Traverso's Isogeny Conjecture predicts a bound on $b_X$ in terms of its height and dimension. The conjecture was proved in \cite{TraversoI} and a refinement was given in \cite{TraversoII}; the latter also gives a bound on $q_X$. In this paper, we study these invariants when restricting to the class of supersingular unitary $p$-divisible groups $X$ of signature $(a,g-a)$, where $0 < a < g$, for $p > 2$, and prove further refinements on their bounds purely in terms of $a$. In particular, we provide a complete description of the possible values of $q_X$. Finally, we rephrase our results in the language of the $\bt_m$ stratifications of characteristic-$p$ fibers of Shimura varieties for $\mathsf{GU}(a, g-a)$ in the case of an inert prime $p > 2$.
\end{abstract}

\maketitle

\section{Introduction}\label{sec:intro}

Given a $p$-divisible group $X$ over $\kk$, an algebraically closed field of characteristic $p$, there are two simpler objects frequently used to capture partial information about $X$: its $p$-torsion subgroup $X[p]$ and its Newton polygon.

In some cases, such as when $X$ is the $p$-divisible group of an elliptic curve, the $p$-torsion group is enough to determine the Newton polygon. In general, it is known that the Newton polygon is determined by some $p$-power torsion. The minimal such number, denoted $b_X$, is called the \emph{isogeny cutoff} (cf. \cref{def:isogney cutoff}). Let the dimension and height of $X$ be denoted as $d_X$ and $h_X$ respectively; Traverso's Isogeny Conjecture \cite[Conj. 5]{Traverso} predicts the upper bound $\lceil d_X(h_X - d_X)/h_X \rceil$ on $b_X$. Traverso's conjecture was proven by Nicole and Vasiu \cite[Theorem 1.2]{TraversoI}, with a further refinement in \cite[Theorem 1.2]{TraversoII} in terms of the Newton polygon of $X$.

The papers \cite{TraversoI,TraversoII} also investigate properties of the \emph{minimal height}, denoted $q_X$, of a $p$-divisible group $X$ (cf. \cref{def:minimal height}); a bound on $q_X\geq 0$ is given in \cite[Theorem 1.5]{TraversoII}. The minimal height measures the size of the kernel of an isogeny between $X$ and the closest minimal $p$-divisible group in the sense of Oort \cite{Oortmin}, it may also be thought of as a measure of distance. The minimal height can be used to give an upper bound on the isogeny cutoff, as explained in \cite[Remark 6.4]{TraversoII}, but it is also an interesting geometric invariant of $X$ in its own right.

The objective of this paper is to study both the minimal heights and the isogeny cutoffs of $p$-divisible groups arising from points in the supersingular locus of characteristic $p$ unitary Shimura varieties. The integral models defined by Kottwiz \cite{kottwitz1992points} of such unitary Shimura varieties exist only when $p > 2$; therefore we restrict to and fix an odd prime $p$. Our unitary Shimura varieties parametrize abelian varieties enjoying extra symmetries, making them a good source of $p$-divisible groups with interesting properties. In addition, unitary Shimura varieties are relevant to many active areas in arithmetic geometry, such as the Kudla--Rapoport program \cite{KRGlobal,KRproof}. The supersingular locus, the unique closed Newton stratum, plays an especially significant role in the Kudla--Rapoport program, as well as other applications. For this reason, in this paper, we focus on the minimal heights and isogeny cutoffs of $p$-divisible groups arising from the supersingular locus.

Unitary $p$-divisible groups (cf. \cref{def:updiv}) come equipped with a signature, which is a pair of positive integers $(a,g-a)$. They fit into the larger collection of $p$-divisible groups of dimension $g$ and height $2g$. For such a a $p$-divisible group $X$, \cite[Theorems 1.2 and 1.5]{TraversoII} give $q_X \leq \lfloor \frac{g}{2} \rfloor$ and $b_X \leq \lceil \frac{g}{2} \rceil$.

For a unitary $p$-divisible group $X$ over $\kk$ of signature $(a,g-a)$, we may assume, without loss of generality, that $a \leq g - a$, or equivalently $g \geq 2a$. Morally, the complexity of $X$ grows with $\min(a,g-a) = a$. With that notation, we give the following complete description of the minimal heights of supersingular unitary $p$-divisible groups.

\begin{introthm}\label{thm:allabtminht}
Define $q_{\max} \coloneqq \begin{cases}
    a & \text{if $g > 2a$;}\\
    a - 1 &\text{if $g = 2a$.}
\end{cases}$

For any supersingular unitary $p$-divisible group $X$ over $\kk$ of signature $(a,g-a)$, we have $q_X\leq q_{\max}$. Furthermore, for any $0\leq q \leq q_{\max}$, there exists a unitary supersingular $p$-divisible group $Y$ over $\kk$ of signature $(a,g-a)$ with $q_Y=q$.
\end{introthm}

\cite[Lemma 6.3 and Remark 6.4]{TraversoII} establishes $b_X \leq 1 + q_X$ and therefore in \cref{thm:isogenycutoffs}, as an application of \cref{thm:allabtminht}, we note that,  
\[b_X \leq 1 + q_{\max} = \begin{cases}
    a + 1 & \text{if $g > 2a$;}\\
    a &\text{if $g = 2a$.}
\end{cases}\]
In \cref{thm:isogenycutoffs} we also exhibit an example of $p$-divisible groups $X$ with $b_X = 1 + q_{\max}$ (cf. \cref{ex:Mk-notuni}), illustrating that this bound for $b_X$ is sharp.

It is useful to note how much the extra symmetry imposed by the signature $(a,g-a)$ action (and compatible quasi-polarization) influence $q_X$ and $b_X$ of $X$. \cref{thm:allabtminht} improves the bound on $q_X$ predicted in \cite{TraversoII} by $\lfloor (g-2a)/2\rfloor$ when $g > 2a$, and by $1$ when $g = 2a$. Thereby the bound on $b_X$ is improved by $\lceil (g-2a)/2\rceil - 1$ when $g > 2a$ as well (and by $0$ when $g = 2a$). These gains grow linearly in $g$ for $g \gg 2a$.

This paper is organized as follows. In \cref{sec:background}, we introduce some background information on unitary $p$-divisible groups and their covariant $p$-adic Dieudonn\'e modules. \cref{sec:chains} is the technical heart of the paper, in which we use explicit calculations with Dieudonn\'e modules to prove several technical lemmas. In \cref{sec:heights}, we apply these lemmas and a series of important examples to prove our main theorem on minimal heights. In \cref{sec:isogenythm}, we give our results on isogeny cutoffs. Finally, in \cref{sec:SVandBTn} we rephrase the results of the previous section in the language of the $\bt_m$ stratification of unitary Shimura varieties, in the sense of \cite{vasiu2006level} where this is referred to as the \emph{level-$m$ stratification}. In particular, we show in \cref{cor:union} that the supersingular locus of a unitary Shimura variety $\mathcal{M}(a,g-a)$ is a union of $\bt_{a+1}$ strata when $g > 2a$, and is a union of $\bt_a$ strata when $g = 2a$.

\subsection*{Acknowledgments}

The collaboration was supported as part of the American Institute of Mathematics (AIM) SQuaRE program. The authors thank AIM and the NSF for their support through this unique and valuable program. We also thank the anonymous referee for their careful consideration and valuable suggestions that helped improve the exposition and results in the paper.

The fourth named author was supported by NSF grant DMS-2201085.

\section{Background}\label{sec:background}

In this section, we give the definition of our main objects of study, unitary $p$-divisible groups, and their linear algebraic realizations, unitary Dieudonn\'e modules. We also recall some useful results from the literature on the structure of these Dieudonn\'e modules.

Throughout this paper, $\kk$ denotes an algebraically closed field of characteristic $p > 2$. Let $W(\kk)$ denote the ring of $p$-typical Witt vectors with coefficients in $\kk$, which is a complete DVR with uniformizer $p$ and residue field $\kk$. Let $W(\kk)_{\qq} \coloneqq W(\kk)[1/p]$ denote the fraction field of $W(\kk)$, and let $\sigma$ denote the Frobenius automorphism of $W(\kk)$ (or $W(\kk)_{\qq}$), which is a lift of the Frobenius automorphism of $\kk$. For any $W(\kk)$-module $M$, we let $M_{\qq} \coloneqq M \otimes_{W(\kk)} W(\kk)_{\qq} = M[1/p]$. 

Given a containment $A \subseteq B$ of two $W(\kk)$-lattices in a $W(\kk)_{\qq}$-vector space, we write $A \subseteq_n B$ to mean that the $W(\kk)$-module $B/A$ has length $n$; we call $n$ the index of $A$ in $B$.

Let $K$ be the unique degree 2 unramified extension of $\qq_p$, with ring of integers $\mathcal{O}_K$. Let $\varphi_0$ and $\varphi_1$ denote the two embeddings of $K$ into $W(\kk)_{\qq}$. Let $a,g$ be integers with $0 < a < g$. Following \cite{VollaardWedhorn}, we make the following definitions. 

\begin{definition}\label{def:updiv}
 A \textbf{unitary $\boldsymbol{p}$-divisible group of signature $\boldsymbol{(a,g-a)}$} over $\kk$  is a tuple $(X, \iota, \lambda)$  where

 \begin{itemize}
\item{$X$ is a $p$-divisible group over $\kk$, of dimension $g$ and height $2g$;}
\item{$\iota: \mathcal{O}_K \rightarrow \mathrm{End}(X)$ is an action satisfying the \emph{signature $(a,g-a)$ condition}: for all $m \in \mathcal{O}_K$,
$$\mathrm{charpol}(\iota(m) \mid \mathrm{Lie}(X) ) = (T - \varphi_0(m))^a(T-\varphi_1(m))^{g-a} \in W(\kk)[T];$$

}
\item{$\lambda: X \rightarrow X^\vee$ is a principal quasi-polarization, where $X^\vee$ is the dual $p$-divisible group, meeting the following $\mathcal{O}_K$-linearity condition, for all $\beta \in \mathcal{O}_K$:
$$\lambda \circ \iota (\beta) = \iota (\overline{\beta})^\vee \circ \lambda.$$
where $\beta \mapsto \overline{\beta}$ is the unique non-trivial involution in $\mathrm{Gal}(K/\qq_p)$. 
}
\end{itemize}

When no confusion is possible, we'll refer to this tuple only as $X$.
    
\end{definition}

\begin{definition}
A \textbf{unitary Dieudonn\'e module of signature $\boldsymbol{(a,g-a)}$} over the field $\kk$  is a tuple $(M, F, V,  \langle  \cdot ,  \cdot \rangle, M = M_0 \oplus M_1 ) $ where

\begin{itemize}
\item{$M$ is a free $W(\kk)$-module of rank $2g$;}
\item{$F: M \rightarrow M$ is a $\sigma$-linear operator and $V: M \rightarrow M$ is a $\sigma^{-1}$-linear operator, i.e., 
\[F(\alpha x) = \sigma(\alpha) F(x) \quad \text{and} \quad V(\alpha x) = \sigma^{-1}(\alpha) V(x), \qquad \text{for all $\alpha \in W(\kk)$ and $x \in M$} ;\]
and $F \circ V = V \circ F = p$;}

\item{$\langle \cdot , \cdot \rangle: M \times M \rightarrow W(\kk)$ is a perfect alternating $W(\kk)$-bilinear pairing on $M$ such that $\langle F(x), y \rangle = \langle x, V(y) \rangle^{\sigma}$ for all $x,y \in M$};
\item{$M = M_0 \oplus M_1$ is a decomposition of $M$ into direct summands of rank $g$, each totally isotropic with respect to $\langle \cdot , \cdot \rangle $, with the property that $F$ and $V$ are grade-reversing operators, i.e., $FM_{i} \subseteq M_{1 - i}$ and $VM_{i} \subseteq M_{1 - i}$ for $i \in \{0,1\}$}, such that 
$$pM_0 \subseteq_{g-a} FM_1 \subseteq_a M_0 \quad \text{  and  } \quad  pM_1 \subseteq_{a} FM_0 \subseteq_{g-a} M_1.$$ 

\end{itemize}

When no confusion is possible, we'll refer to this tuple only as $M$.
\end{definition}

 As discussed in \cite{VollaardWedhorn}, covariant Dieudonn\'e theory gives a bijection between the set of all unitary $p$-divisible groups of signature $(a,g-a)$ over $\kk$ and the set of all unitary Dieudonn\'e modules of signature $(a,g-a)$ over $\kk$. 

Note that reversing the labeling of the embeddings $\varphi_0$ and $\varphi_1$, or on the level of Dieudonn\'e modules the labeling of $M_0$ and $M_1$, changes the signature from $(a,g-a)$ to $(g-a,a)$. So, without loss of generality we always assume $a \leq g - a$, so that $\min(a,g-a)=a.$

Thus, we fix the following notational convention throughout this paper:
\begin{notation}\label{not:a-n-g}
    Let $a$ and $g$ be integers satisfying $0 < a < g$ and $a \leq g - a$, or equivalently, $g \geq 2a$.
\end{notation}

\cref{not:a-n-g} is fixed for technical reasons concerning the relevant Shimura varieties, but it is convenient in some instances, when mathematically sound, to allow the degenerate cases $a = 0$ or $a = g > 0$.

Recall that a $p$-divisible group over $\mathbx{k}$ is called $\emph{supersingular}$ (resp. \emph{superspecial}) if it is isogenous (resp. isomorphic) to the $p$-divisible group of a product of supersingular elliptic curves over $\mathbx{k}$. We also refer to the corresponding $p$-adic Dieudonn\'e module as supersingular (resp. superspecial).

Let $(M, F, V,  \langle  \cdot ,  \cdot \rangle, M = M_0 \oplus M_1 ) $ be a unitary Dieudonn\'e module over $\kk$ of signature $(a,g-a)$. Due to the compatibility between the action and the quasi-polarization, much of the structure of $M$ is controlled by the submodule $M_0$ (or, symmetrically, $M_1$). To see this, for $i \in \{0,1\}$, it is possible to define a new pairing on $M_i$, as 
\begin{align*}
    \llangle \cdot , \cdot \rrangle: M_i \times M_i &\rightarrow W(\kk)\\
\llangle   x , y  \rrangle &\coloneqq \langle x, F(y) \rangle.
\end{align*}

We use the same notation to denote the extension of scalars of this pairing to a $W(\kk)_{\qq}$-valued pairing on $(M_i)_{\qq}$. Given a $W(\kk)$-lattice $L \subseteq (M_i)_{\qq}$, we denote 
$$L^\vee = \{ x \in (M_i)_{\qq} \ | \ \llangle x, y \rrangle \in W(\kk) \text{ for all } y \in L \}.$$

Let $\tau$ be the $\sigma^2$-linear operator $p^{-1}F^2 =  V^{-1} \circ F$ of $M_{\qq}$. As $F$ and $V$ interchange $M_0$ and $M_1$, note that $\tau$ is an operator on each $(M_i)_{\qq}$. Some useful properties of $\tau$ and the pairing $\llangle \cdot , \cdot \rrangle$, due to Vollaard \cite{Vollaard}, are recorded below.

\begin{proposition}\label{prop:0and1conditions}
Let $(M, F, V,  \langle  \cdot ,  \cdot \rangle, M = M_0 \oplus M_1 ) $ be a unitary Dieudonn\'e module over $\kk$ of signature $(a,g-a)$. For $i\in\{0,1\}$ and any $W(\kk)$-lattice $L \subseteq (M_i)_{\qq}$, we have the following:
\begin{enumerate}[label=\textnormal{(\alph*)}]
\item{ $\tau(L) ^\vee = \tau(L^\vee)$;}
\item{$(L^\vee)^\vee = \tau(L)$;}
\item{$F (M_i) = p M_{1-i}^\vee$, or equivalently $M_i = F^{-1} (p M_{1-i}^\vee)$.} 
\end{enumerate}
The relationship between $M_0$ and $M_1$ can be expressed in the following four (equivalent) chain conditions:
\begin{align*}
pM_0^\vee & \subseteq_a  M_0 \subseteq_{g-a} M_0^\vee, \hspace{1cm}
pM_0^\vee \subseteq_a \tau(M_0) \subseteq_{g-a} M_0^\vee, \\
pM_1^\vee &\subseteq_{g-a} M_1 \subseteq_a M_1^\vee, \; \text{ and } \;pM_1^\vee \subseteq_{g-a} \tau(M_1) \subseteq_a M_1^\vee.   
\end{align*}

\end{proposition}

\begin{proof}
These statements can be found in  \cite[Subsection 1.11, Proposition 1.12, and Remark 1.13]{Vollaard}.  Note that \cite{Vollaard} is concerned only with the Dieudonn\'e modules of \emph{supersingular} $p$-divisible groups, but that assumption is not relevant to these statements.
\end{proof}

\section{Isogenies to Superspecial $p$-Divisible Groups}\label{sec:chains}

In this section, for any supersingular unitary $p$-divisible group $X$ of signature $(a,g-a)$, we construct a superspecial $p$-divisible group $\uxx$ and an isogeny $\rho: X \rightarrow \uxx$ with the property that $\mathrm{ker}(\rho) \subseteq X[p^a]$  (in fact, contained in  $X[p^{a-1}]$ when $g = 2a$). This result is in \cref{cor:isog}, which relies on a series of lemmas concerning the action of $\tau$ on $M_0$ and $M_1$. 

Consider a supersingular unitary $p$-divisible group $X$ over $\kk$ of signature $(a,g-a)$, with Dieudonn\'e module $(M, F, V,  \langle  \cdot ,  \cdot \rangle, M = M_0 \oplus M_1 ) $. For each $i \in \zz_{\geq 0}$, we define $W(\kk)$-lattices: 
\begin{align*}
T_i \coloneqq M_0 + \tau(M_0) + \cdots + \tau^i(M_0) \; 
\text{ and }\; S_i \coloneqq M_1 + \tau(M_1) + \cdots + \tau^i(M_1);
\end{align*}
in particular, $T_0 = M_0$ and $S_0 = M_1$ and $T_i = T_{i-1} + \tau(T_{i-1})$ (similarly for $S_i$).

\begin{lemma}\label{lem:existenceofmandn}
Let $M$ be a unitary $p$-divisible group of signature $(a,g-a)$ over $\kk$. Assume that $M$ is supersingular. There exist non-negative integers $m$ and $n$ such that $T_{m}$ and $S_{n}$ are $\tau$-invariant.
\end{lemma}

\begin{proof}
Note that because $M$ is supersingular, the operator $F$ is a slope $\frac{1}{2}$ operator on the isocrystal $M_{\qq}$. So, the operator $\tau$ on both $(M_0)_{\qq}$ and $(M_1)_{\qq}$ is slope zero. Then, the existence of $m$ and $n$ follows from \cite[Proposition 2.17]{rapoport1996period}. 
\end{proof}

\begin{notation}\label{not:min-tau-inv}
    Let $\mu$ and $\nu$ be the smallest non-negative integers such that $T_\mu$ and $S_\nu$ are $\tau$-invariant.
\end{notation}

Observe that, by definition of $\mu$, one has $\tau(T_{{\mu+i}}) = T_{{\mu+i}}$ for any $i \in \zz_{\geq 0}$, and therefore we obtain $T_{{\mu+i+1}} = T_{{\mu+i}} + \tau(T_{{\mu+i}}) = T_{{\mu+i}}$. Hence, we have $T_{{\mu}} = T_{{\mu+1}} = T_{{\mu+2}} = \cdots .$ A similar observation holds for $S_\nu$.

In order to understand the relationship between $T_\mu$ and $T_0 = M_0$, it is important to understand the relationship between each $T_i$ and $T_{i+1}$ (and similarly for  $S_\nu$ and $S_0 = M_1$). It is convenient to introduce the following notation and the diagrams in \cref{fig:Tidiag,fig:Sidiag} to understand these relationships.

Set $T_{-1} \coloneqq \tau^{-1}(p M_0^\vee)$ and $S_{-1} \coloneqq \tau^{-1}(p M_1^\vee).$ 
We define four sequences of integers: 
\begin{itemize}
\item{$c_i$, such that $T_{i-1}  \subseteq_{c_i} T_{i}$ for $i \in \{0,1,..., \mu+1\}$;}
\item{$g_i$, such that $S_{i-1}  \subseteq_{g_i} S_{i}$ for $i \in \{0,1,..., \nu+1\}$;}
\item{$d_i$, such that $\tau(T_{i-1}) \subseteq_{d_i} T_i \cap \tau(T_i)$ for $i \in \{0,1,..., \mu\}$;}
\item{$h_i$, such that $\tau(S_{i-1}) \subseteq_{h_i} S_i \cap \tau(S_i)$  for $i \in \{0,1,..., \nu\}$.}
\end{itemize}

The $i = 0$ case, $T_{-1} \subseteq T_0$ (resp. $S_{-1} \subseteq S_0$), is well-defined by \cref{prop:0and1conditions} which also gives us $c_0 = a$ (resp. $g_0 = g - a$). Moreover, by minimality of $\mu$ and $\nu$ we have $c_{\mu + 1} = g_{\nu+1} = 0$.

\begin{figure}[h!]
\centering
\[\begin{adjustbox}{max height=.2\textheight}
\begin{tikzcd}
	\cdots & {\tau^{-2}(T_\mu)} && {\tau^{-1}(T_\mu)} && {T_\mu} && {\tau(T_\mu)} & \cdots \\
	&& {\tau^{-2}(T_\mu)\cap \tau^{-1}(T_\mu)} && {\tau^{-1}(T_\mu)\cap (T_\mu)} && {(T_\mu)\cap \tau(T_\mu)} \\
	& \cdots & {\tau^{-1}(T_{\mu-1})} && {T_{\mu-1}} && {\tau(T_{\mu-1})} & \cdots \\
	&& \vdots && \vdots && \vdots \\
	&\cdots& {\tau^{-1}(T_2)} && {T_2} && {\tau(T_2)} & \cdots \\
	&&& {\tau^{-1}(T_2) \cap (T_2)} && {T_2 \cap \tau(T_2)} \\
	\cdots & {\tau^{-1}(T_1)} && {T_1} && {\tau(T_1)} && {\tau^2(T_1)} & \cdots \\
	&& {\tau^{-1}(T_1) \cap T_1} && {T_1 \cap \tau(T_1)} && {\tau(T_1) \cap \tau^2(T_1)} \\
	& \cdots & {T_0} && {\tau(T_0)} && {\tau^2(T_0)} & \cdots \\
	&&& {T_0 \cap \tau(T_0)} && {\tau(T_0) \cap \tau^2(T_0)} \\
	\cdots & {T_{-1}} && {\tau(T_{-1})} && {\tau^2(T_{-1})} && {\tau^3(T_{-1})} & \cdots .\\
	\arrow[no head, from=1-2, to=2-3]
	\arrow[no head, from=1-2, to=3-3]
	\arrow[no head, from=1-4, to=2-3]
	\arrow[no head, from=1-4, to=2-5]
	\arrow[no head, from=1-4, to=3-3]
	\arrow[no head, from=1-4, to=3-5]
	\arrow[no head, from=1-6, to=2-5]
	\arrow["{c_{\mu+1}=0}", no head, from=1-6, to=2-7]
	\arrow["{\boldsymbol{c_{\mu}}}", no head, from=1-6, to=3-5]
	\arrow["{c_{\mu}}"', no head, from=1-6, to=3-7]
	\arrow[no head, from=1-8, to=2-7]
	\arrow[no head, from=1-8, to=3-7]
	\arrow[no head, from=2-3, to=3-3]
	\arrow[no head, from=2-5, to=3-5]
	\arrow["{d_{\mu}}"{pos=0.4}, no head, from=2-7, to=3-7]
	\arrow[no head, from=5-3, to=6-4]
	\arrow[no head, from=5-3, to=7-2]
	\arrow[no head, from=5-3, to=7-4]
	\arrow[no head, from=5-5, to=6-4]
	\arrow["{c_3}", no head, from=5-5, to=6-6]
	\arrow["{\boldsymbol{c_2}}", no head, from=5-5, to=7-4]
	\arrow["{c_2}"', no head, from=5-5, to=7-6]
	\arrow[no head, from=5-7, to=6-6]
	\arrow[no head, from=5-7, to=7-6]
	\arrow[no head, from=5-7, to=7-8]
	\arrow[no head, from=6-4, to=7-4]
	\arrow["{d_2}", no head, from=6-6, to=7-6]
	\arrow[no head, from=7-2, to=8-3]
	\arrow[no head, from=7-2, to=9-3]
	\arrow[no head, from=7-4, to=8-3]
	\arrow["{c_2}", no head, from=7-4, to=8-5]
	\arrow["{\boldsymbol{c_1}}", no head, from=7-4, to=9-3]
	\arrow["{c_1}"', no head, from=7-4, to=9-5]
	\arrow["{c_2}"', no head, from=7-6, to=8-5]
	\arrow[no head, from=7-6, to=8-7]
	\arrow[no head, from=7-6, to=9-5]
	\arrow[no head, from=7-6, to=9-7]
	\arrow[no head, from=7-8, to=8-7]
	\arrow[no head, from=7-8, to=9-7]
	\arrow[no head, from=8-3, to=9-3]
	\arrow["{d_1}", no head, from=8-5, to=9-5]
	\arrow[no head, from=8-7, to=9-7]
	\arrow["{c_1}", no head, from=9-3, to=10-4]
	\arrow["{\boldsymbol{c_0=a}}", no head, from=9-3, to=11-2]
	\arrow["{c_0=a}"', no head, from=9-3, to=11-4]
	\arrow["{c_1}"', no head, from=9-5, to=10-4]
	\arrow[no head, from=9-5, to=10-6]
	\arrow["a", no head, from=9-5, to=11-4]
	\arrow[no head, from=9-5, to=11-6]
	\arrow[no head, from=9-7, to=10-6]
	\arrow[no head, from=9-7, to=11-6]
	\arrow[no head, from=9-7, to=11-8]
	\arrow["{d_0}", no head, from=10-4, to=11-4]
	\arrow[no head, from=10-6, to=11-6]
\end{tikzcd}
\end{adjustbox}\]
\caption{$T_i$ Diagram}
  \label{fig:Tidiag}
\end{figure}
  
\begin{figure}[h!]
\centering
\[\begin{adjustbox}{max height=.2\textheight}
\begin{tikzcd}
	\cdots & {\tau^{-2}(S_\nu)} && {\tau^{-1}(S_\nu)} && {S_\nu} && {\tau(S_\nu)} & \cdots \\
	&& {\tau^{-2}(S_\nu)\cap \tau^{-1}(S_\nu)} && {\tau^{-1}(S_\nu)\cap (S_\nu)} && {(S_\nu)\cap \tau(S_\nu)} \\
	& \cdots & {\tau^{-1}(S_{\nu-1})} && {S_{\nu-1}} && {\tau(S_{\nu-1})} & \cdots \\
	&& \vdots && \vdots && \vdots \\
	& \cdots & {\tau^{-1}(S_2)} && {S_2} && {\tau(S_2)} & \cdots \\
	&&& {\tau^{-1}(S_2) \cap (S_2)} && {S_2 \cap \tau(S_2)} \\
	\cdots & {\tau^{-1}(S_1)} && {S_1} && {\tau(S_1)} && {\tau^2(S_1)} & \cdots \\
	&& {\tau^{-1}(S_1) \cap S_1} && {S_1 \cap \tau(S_1)} && {\tau(S_1) \cap \tau^2(S_1)} \\
	& \cdots & {S_0} && {\tau(S_0)} && {\tau^2(S_0)} & \cdots \\
	&&& {S_0 \cap \tau(S_0)} && {\tau(S_0) \cap \tau^2(S_0)} \\
	\cdots & {S_{-1}} && {\tau(S_{-1})} && {\tau^2(S_{-1})} && {\tau^3(S_{-1})} & \cdots .\\
	\arrow[no head, from=1-2, to=2-3]
	\arrow[no head, from=1-2, to=3-3]
	\arrow[no head, from=1-4, to=2-3]
	\arrow[no head, from=1-4, to=2-5]
	\arrow[no head, from=1-4, to=3-3]
	\arrow[no head, from=1-4, to=3-5]
	\arrow[no head, from=1-6, to=2-5]
	\arrow["{g_{\nu+1}=0}", no head, from=1-6, to=2-7]
	\arrow["{\boldsymbol{g_{\nu}}}", no head, from=1-6, to=3-5]
	\arrow["{g_{\nu}}"', no head, from=1-6, to=3-7]
	\arrow[no head, from=1-8, to=2-7]
	\arrow[no head, from=1-8, to=3-7]
	\arrow[no head, from=2-3, to=3-3]
	\arrow[no head, from=2-5, to=3-5]
	\arrow["{h_{\nu}}"{pos=0.4}, no head, from=2-7, to=3-7]
	\arrow[no head, from=5-3, to=6-4]
	\arrow[no head, from=5-3, to=7-2]
	\arrow[no head, from=5-3, to=7-4]
	\arrow[no head, from=5-5, to=6-4]
	\arrow["{g_3}", no head, from=5-5, to=6-6]
	\arrow["{\boldsymbol{g_2}}", no head, from=5-5, to=7-4]
	\arrow["{g_2}"', no head, from=5-5, to=7-6]
	\arrow[no head, from=5-7, to=6-6]
	\arrow[no head, from=5-7, to=7-6]
	\arrow[no head, from=5-7, to=7-8]
	\arrow[no head, from=6-4, to=7-4]
	\arrow["{h_2}", no head, from=6-6, to=7-6]
	\arrow[no head, from=7-2, to=8-3]
	\arrow[no head, from=7-2, to=9-3]
	\arrow[no head, from=7-4, to=8-3]
	\arrow["{g_2}", no head, from=7-4, to=8-5]
	\arrow["{\boldsymbol{g_1}}", no head, from=7-4, to=9-3]
	\arrow["{g_1}"', no head, from=7-4, to=9-5]
	\arrow["{g_2}"', no head, from=7-6, to=8-5]
	\arrow[no head, from=7-6, to=8-7]
	\arrow[no head, from=7-6, to=9-5]
	\arrow[no head, from=7-6, to=9-7]
	\arrow[no head, from=7-8, to=8-7]
	\arrow[no head, from=7-8, to=9-7]
	\arrow[no head, from=8-3, to=9-3]
	\arrow["{h_1}", no head, from=8-5, to=9-5]
	\arrow[no head, from=8-7, to=9-7]
	\arrow["{g_1}", no head, from=9-3, to=10-4]
	\arrow["{\boldsymbol{g_0=g-a}}", no head, from=9-3, to=11-2]
	\arrow["{g_0=g-a}"', no head, from=9-3, to=11-4]
	\arrow["{g_1}"', no head, from=9-5, to=10-4]
	\arrow[no head, from=9-5, to=10-6]
	\arrow["{g-a}", no head, from=9-5, to=11-4]
	\arrow[no head, from=9-5, to=11-6]
	\arrow[no head, from=9-7, to=10-6]
	\arrow[no head, from=9-7, to=11-6]
	\arrow[no head, from=9-7, to=11-8]
	\arrow["{h_0}", no head, from=10-4, to=11-4]
	\arrow[no head, from=10-6, to=11-6]
\end{tikzcd}
\end{adjustbox}\]
\caption{$S_i$ Diagram}
  \label{fig:Sidiag}
\end{figure}

The relations between the lattices $T_i$ (resp. $S_i$) under the action of $\tau$ and intersections are demonstrated in the diagram in \cref{fig:Tidiag} (resp. \cref{fig:Sidiag}). In all diagrams in this section, an edge connecting a lattice $A$ to a lattice $B$, where $B$ is above $A$, is communicating the information that $A \subseteq B$. If such an edge is labeled $n$, then this communicates the information that $A\subseteq_n B$.

Note that the containments $\tau(T_{-1}) \subseteq_{a} T_{0}$ and $\tau(T_{-1}) \subseteq_{a} \tau(T_{0})$ are due to \cref{prop:0and1conditions}, and then the other containments shown in \cref{fig:Tidiag} follow from the definition of the $T_{i}$ and the fact that $\tau$ is a $\sigma^2$-linear automorphism of the isocrystal $(M_i)_{\qq}$.

Say a containment of $W(\kk)$-lattices $A \subseteq B$ is \emph{adjacent} if $pB \subseteq A \subseteq B$. If $A \subseteq B \subseteq C$ and $A \subseteq C$ is an adjacent containment, then $A \subseteq B$ and $B \subseteq C$ are also adjacent. We claim that all edges in \cref{fig:Tidiag,fig:Sidiag} communicate adjacency of lattices.

The labels and adjacency in \cref{fig:Tidiag} are justified by \cref{lem:latt-diag-ind} below. The following fact is used in the lemma: given $W(\kk)$-lattices $A \subseteq B \subseteq C$, one has $[C:A] = [C:B] + [B:A]$, since length is additive over a short exact sequence which we obtain from the third isomorphism theorem $C/B \cong (C/A)/(B/A)$.

\begin{lemma}\label{lem:latt-diag-ind}
For any $0 \leq i \leq \mu$, we have the following:
\begin{enumerate}[label=\textnormal{(\alph*)}]
    \item \label{item:latt-ind} $[T_{i}:T_{i} \cap \tau(T_{i})] = [\tau(T_{i}):T_{i} \cap \tau(T_{i})]$, and consequently 
    \[c_{i+1} = [T_{i+1}:T_{i}] = [T_{i+1}:\tau(T_{i})] = [T_{i}:T_{i} \cap \tau(T_{i})] = [\tau(T_{i}):T_{i} \cap \tau(T_{i})].\]
    \item \label{item:ind-reln} $c_i = c_{i+1} + d_i$, and consequently, $\displaystyle c_i = \sum_{j = i}^\mu d_j$;
    in particular, $c_\mu = d_\mu$.
    \item \label{item:lat-adj} Every containment displayed in \cref{fig:Tidiag} is adjacent.
\end{enumerate}
\end{lemma}

\begin{proof}
We prove \ref{item:latt-ind} by inducting on $i$. Consider $i = 0$. We have observed, using \cref{prop:0and1conditions}, that $[T_{0}:\tau(T_{-1})] = [\tau(T_{0}):\tau(T_{-1})] = c_{0}$ and $\tau(T_{-1}) \subseteq T_{0} \cap \tau(T_{0})$. Therefore, we have the following diagram

\[{\footnotesize\begin{tikzcd}[column sep=large, row sep=small]
T_0 \arrow[rdd, "c_0"', no head] \arrow[rd, no head] &                                       & \tau(T_0) \arrow[ldd, "c_0", no head] \arrow[ld, no head] \\
                                                     & T_0 \cap \tau(T_0) \arrow[d, no head] &                                                           \\[1em]
                                                     & \tau(T_{-1})                          &                                                          
. \end{tikzcd}}\]
By additivity of length, we have \[ [T_{0}:T_{0} \cap \tau(T_{0})] + [T_{0} \cap \tau(T_{0}):\tau(T_{-1})] = c_0 = [\tau(T_{0}):T_{0} \cap \tau(T_{0})] + [T_{0} \cap \tau(T_{0}):\tau(T_{-1})],\]
giving us $[T_{0}:T_{0} \cap \tau(T_{0})] = [\tau(T_{0}):T_{0} \cap \tau(T_{0})]$. 

Then, by the second isomorphism theorem,
\[T_1/T_{0} = (T_{0}+ \tau(T_{0}))/T_{0} \cong \tau(T_{0})/(T_{0}\cap \tau(T_{0}))\]
and
\[T_1/\tau(T_{0}) = (T_{0}+ \tau(T_{0}))/\tau(T_{0}) \cong T_{0}/(T_{0}\cap \tau(T_{0}))\]
Hence,
\[c_1 = [T_1:T_{0}] = [\tau(T_0):T_0 \cap \tau(T_0)] = [T_0:T_0 \cap \tau(T_0)] = [T_1:\tau(T_{0})],\]
proving the base case. 

Assume \[ c_{i+1} =  [T_{i+1}:T_{i}] = [T_{i+1}:\tau(T_{i})] = [T_{i}:T_{i} \cap \tau(T_{i})] = [\tau(T_{i}):T_{i} \cap \tau(T_{i})]\] is true for some $i \geq 0$. Note that $\tau(T_i) \subseteq T_{i+1} \cap \tau(T_{i+1})$ directly from the definition of the $T_i$. From our inductive hypothesis, $c_{i+1} = [T_{i+1}:\tau(T_{i})]$. Since $\tau$ is a $\sigma^2$-linear automorphism and, by definition, $c_{i+1} =  [T_{i+1}:T_{i}] $, we obtain $c_{i+1} =  [\tau(T_{i+1}):\tau(T_{i})]$. Hence we have the following diagram

\[{\footnotesize\begin{tikzcd}[column sep=large, row sep=small]
T_{i+1} \arrow[rdd, "c_{i+1}"', no head] \arrow[rd, no head] &                                       & \tau(T_{i+1}) \arrow[ldd, "c_{i+1}", no head] \arrow[ld, no head] \\
                                                     & T_{i+1} \cap \tau(T_{i+1}) \arrow[d, no head] &                                                           \\[1em]
                                                     & \tau(T_{i})                          &                                                          
. \end{tikzcd}}\]

Just as in the base case, by additivity of length, it follows that 
$$[T_{i+1}: T_{i+1} \cap \tau(T_{i+1} )] = [ \tau(T_{i+1} ) : T_{i+1} \cap \tau (T_{i+1}) ].$$

Then, by applying the second isomorphism theorem, 
we have
\[T_{i+2}/T_{i+1} = (T_{i+1}+ \tau(T_{i+1}))/T_{i+1} \cong \tau(T_{i+1})/(T_{i+1}\cap \tau(T_{i+1}))\]
and
\[T_{i+2}/\tau(T_{i+1}) = (T_{i+1}+ \tau(T_{i+1}))/\tau(T_{i+1}) \cong T_{i+1}/(T_{i+1}\cap \tau(T_{i+1}))\]
Hence,
\[c_{i+2} = [T_{i+2}:T_{i+1}] = [\tau(T_{i+1}):T_{i+1} \cap \tau(T_{i+1})] = [T_{i+1}:T_{i+1} \cap \tau(T_{i+1})] = [T_{i+2}:\tau(T_{i+1})],\]
as we wished to show.

For \ref{item:ind-reln}, by \ref{item:latt-ind} and definition of $c_i$ and $d_i$, we have
{\footnotesize
\[\begin{tikzcd}
{T_{i}} \arrow[d, "c_{i+1}", no head] \arrow[dd, "c_i"', no head, bend right=49] \\
{\qquad T_{i} \cap \tau(T_{i})} \arrow[d, "d_i", no head]                      \\
{T_{i-1}}                                                                       
\end{tikzcd}\]
}

giving us $c_{i} = c_{i+1} + d_{i}$. Rewriting this as $c_{i+1} = c_{i} - d_{i}$, we obtain for any $0 \leq i \leq \mu$
\[c_{\mu+1} = c_{\mu} - d_{\mu} = (c_{\mu - 1} - d_{\mu - 1}) - d_{\mu} = \cdots = c_{i} - \sum_{j = i}^{\nu} d_{j}.\]
As observed earlier, minimality of $\mu$ gives us $c_{\mu+1} = 0$ allowing us to conclude $\displaystyle c_{i} = \sum_{j = i}^{\mu} d_{j}$. 

Finally, we focus on \ref{item:lat-adj}. As a base case, note that every containment in the diagram 
\[{\footnotesize\begin{tikzcd}[column sep=large, row sep=small]
T_0 \arrow[rdd,  no head] \arrow[rd, no head] &                                       & \tau(T_0) \arrow[ldd,  no head] \arrow[ld, no head] \\
                                                     & T_0 \cap \tau(T_0) \arrow[d, no head] &                                                           \\[1em]
                                                     & \tau(T_{-1})                          &                                                          
\end{tikzcd}}\]
is an adjacent containment: by Proposition \ref{prop:0and1conditions}, the containments $\tau(T_{-1}) \subseteq T_0$ and $\tau(T_{-1}) \subseteq \tau(T_0)$ are adjacent, which implies the other displayed containments are as well.

Assume for some $i \geq 0$ that all the containments in the diagram 
\[{\footnotesize\begin{tikzcd}[column sep=large, row sep=small]
T_i \arrow[rdd,  no head] \arrow[rd, no head] &                                       & \tau(T_i) \arrow[ldd,  no head] \arrow[ld, no head] \\
                                                     & T_i \cap \tau(T_i) \arrow[d, no head] &                                                           \\[1em]
                                                     & \tau(T_{i-1})                          &                                                          
\end{tikzcd}}\]
are adjacent. We will show that all the containments in the diagram 
\[{\footnotesize\begin{tikzcd}[column sep=large, row sep=small]
T_{i+1} \arrow[rdd,  no head] \arrow[rd, no head] &                                       & \tau(T_{i+1}) \arrow[ldd,  no head] \arrow[ld, no head] \\
                                                     & T_{i+1} \cap \tau(T_{i+1}) \arrow[d, no head] &                                                           \\[1em]
                                                     & \tau(T_{i})                          &                                                          
\end{tikzcd}}\]
are also adjacent. Indeed, by the second isomorphism theorem, $T_{i+1}/ T_i \cong \tau(T_{i})/(T_i \cap \tau(T_i) )$ and $T_{i+1}/\tau(T_i) \cong T_{i}/(T_i \cap \tau(T_i))$, so the containments $T_i \subseteq T_{i+1}$ and $\tau(T_i) \ \subseteq T_{i+1}$ are also adjacent. Then, since applying a $\sigma^2$-linear automorphism preserves adjacency, the containment $\tau(T_i) \subseteq \tau(T_{i+1})$ is adjacent, and it follows that all the containments displayed in the above diagram are adjacent. Then, by induction, every containment in Figure 1 is adjacent, as they are all translates of an adjacent containment by some power of the $\sigma^2$-linear automorphism $\tau$.
\end{proof}

The methods in the proof of \cref{lem:latt-diag-ind} are formal and therefore similar arguments justify and prove the following analogous claims for $S_i$ vis-a-vis \cref{fig:Sidiag}.
\begin{lemma}\label{lem:latt-ind-Si}
For any $0 \leq i \leq \nu$, we have the following:
\begin{enumerate}[label=\textnormal{(\alph*)}]
    \item \label{item:latt-ind-Si} $[S_{i}:S_{i} \cap \tau(S_{i})] = [\tau(S_{i}):S_{i} \cap \tau(S_{i})]$, and consequently 
    \[g_{i+1} = [S_{i+1}:S_{i}] = [S_{i+1}:\tau(S_{i})] = [S_{i}:S_{i} \cap \tau(S_{i})] = [\tau(S_{i}):S_{i} \cap \tau(S_{i})].\]
    \item \label{item:ind-reln-Si} $g_i = g_{i+1} + h_i$, and consequently, $\displaystyle g_i = \sum_{j = i}^\mu h_j$; in particular, $g_\nu = h_\nu$.
    \item \label{item:lat-adj-Si} Every containment displayed in \cref{fig:Sidiag} is adjacent.
\end{enumerate}
\end{lemma}

Using  \cref{fig:Tidiag}, we are now ready to describe the relationship between $M_0$ and $T_\mu$ in \cref{leftlatticelemma} below. We include more detail than is strictly necessary for this paper, so that we can describe a result analogous to \cite[Lemma 2.2]{Vollaard}, \cite[Proposition 2.19]{howard2014supersingular}, \cite[Proposition 4.4]{GL4}, etc., as a potentially useful addition to the literature. 

\begin{lemma}\label{leftlatticelemma}
Let $M$ be a supersingular unitary Dieudonn\'e module over $\kk$ of signature $(a,g-a)$. 
Then we have the following inclusions:
$$p^a T_\mu \subseteq pT_\mu^\vee  \subseteq p M_0 ^\vee \subseteq_a M_0 \subseteq T_\mu \subseteq  p^{1-a} T_\mu^\vee.$$
\end{lemma}

\begin{proof}
We refer to the integers $c_i$ and $d_i$ of \cref{fig:Tidiag}. By \cref{lem:latt-diag-ind} \ref{item:ind-reln} applied to $i = 0$, we have
\[c_0 = \sum_{j=0}^{\mu} d_j\]
Since $c_0 = a > 0$ and $d_i \geq 0$ for any $i$, at most $a$ of $d_0,\ldots,d_{\nu}$ are positive.

If $d_i = 0$, then $\tau(T_{i-1}) = T_i \cap \tau(T_i),$ and so
$$\bigcap_{\ell \in \zz} \tau^\ell(T_i) = \bigcap_{\ell \in \zz} \tau^\ell(T_{i-1}).$$

If $d_i > 0$, then using adjacency we still have  $p ( T_i \cap \tau(T_i) )\subseteq  \tau(T_{i-1})$
and so 
$$ p \bigcap_{\ell \in \zz} \tau^\ell(T_i) \subseteq 
  \bigcap_{\ell \in \zz} \tau^\ell(T_{i-1}).$$

Since  at most $a$ of the integers $d_0, \dots, d_{\mu}$ are strictly greater than $0$, we have that
\begin{equation}\label{eq:eq1}
    p^a  T_\mu = p^a \bigcap_{\ell \in \zz} \tau^\ell(T_\mu) \subseteq   \bigcap_{\ell \in \zz} \tau^\ell(T_{-1}).
\end{equation}

Note that, because $T_\mu$ is $\tau$-invariant
\begin{align*}
T_\mu &= \cdots +  \tau^{-1} (T_\mu) + T_\mu + \tau(T_\mu) + \tau^2 (T_\mu) + \cdots  \\
&= \cdots +  \tau^{-1} (M_0) + M_0 + \tau(M_0) + \tau^2 (M_0) + \cdots. 
\end{align*}

Then,
\begin{align*}
T_\mu^\vee &= \bigcap_{\ell \in \zz }
 \tau^\ell(M_0)^\vee = \bigcap_{\ell \in \zz }
 \tau^\ell(M_0^\vee) =   \bigcap_{\ell \in \zz }
 \tau^\ell( p^{-1} T_{-1})=  p^{-1} \bigcap_{\ell \in \zz }
 \tau^\ell(  T_{-1}     ).\tag{2}\label{eq:eq2}
 \end{align*} 
And so, by \cref{eq:eq1,eq:eq2}, we have
$$T_\mu \subseteq  p^{1-a} T_\mu^\vee.$$

Finally, combining the above with the facts that $p M_0^\vee \subseteq_a M_0$ and $T_\mu^\vee \subseteq M_0^\vee$,
\[p^a T_\mu \subseteq pT_\mu^\vee  \subseteq p M_0 ^\vee \subseteq_a M_0 \subseteq T_\mu \subseteq  p^{1-a} T_\mu^\vee.\qedhere\]
\end{proof}

Repeating the arguments of \cref{leftlatticelemma} with $M_1$, replacing $M_0$, yields a similar chain condition:
$$p^{g-a} S_\nu \subseteq pS_\nu^\vee  \subseteq p M_1 ^\vee \subseteq_{g-a} M_1 \subseteq S_\nu \subseteq  p^{1-b} S_\nu^\vee,$$
which do not suffice for the applications we have in mind for these lemmas. These relations would combine to produce an isogeny to a superspecial $p$-divisible group with kernel contained in $X[p^{g-a}]$, which is in fact no better (and, for almost all signatures, worse) than is already known for $p$-divisible groups without unitary action and quasi-polarization.

However, a key observation does produce the chain condition we need, in \cref{rightlatticelemma} below.

\begin{lemma}\label{rightlatticelemma}
Let $M$ be a supersingular unitary Dieudonn\'e module over $\kk$ of signature $(a,g-a)$. Then,
$$p^a S_\nu  \subseteq M_1 \subseteq S_\nu.$$ 
\end{lemma}

\begin{proof}
We refer to the lattice diagram in \cref{fig:Sidiag} and the notation $g_i$ and $h_i$ shown there.  

Unlike \cref{leftlatticelemma}, in this case we have $g_0=g-a$, not $a$. However, in the proof of that lemma we showed that
$$M_0 \subseteq_{c_1} M_0 + \tau( M_0)$$
with $c_1 \leq a$. Applying the dual followed by the operator $F^{-1}p$ to both sides, we get
$$ F^{-1}(p M_0^\vee) \cap \tau(F^{-1}(p M_0^\vee) )  \subseteq_{c_1} F^{-1}(pM_0^\vee).$$
Since $M_1 = F^{-1}(p M_0^\vee)$, and by the definition of $S_i$, the above expression becomes $S_0 \cap \tau(S_0) \subseteq_{c_1} S_0$. In particular, $g_1 = c_1 \leq a$. We can now proceed similarly to the previous lemma. 

\cref{lem:latt-ind-Si} \ref{item:ind-reln-Si} for $i = 1$, gives us $\sum_{j = 1}^{\nu} h_j = g_1 \leq a$, so at most $a$ of the integers $h_1, \dots, h_\nu$ are greater than zero. Using the fact that all inclusions shown in the diagram  in \cref{fig:Sidiag} are adjacent,
\[p^a  S_\nu = p^a \bigcap_{\ell \in \zz} \tau^\ell(S_\nu) 
 \subseteq   \bigcap_{\ell \in \zz} \tau^\ell(S_0)  
 \subseteq S_0. 
\]

Therefore, $p^a S_\nu \subseteq M_1 \subseteq S_\nu$. \end{proof}

When $g = 2a$, the results in \cref{leftlatticelemma,rightlatticelemma} can be improved slightly.

\begin{lemma}\label{parallel}

Let $M$ be a supersingular unitary Dieudonn\'e module over $\kk$ of signature $(a,a)$. Then, 
$$p^{a-1} T_\mu \subseteq  M_0 \subseteq T_\mu  \quad \text{ and } \quad p^{a-1} S_\nu  \subseteq M_1 \subseteq S_\nu.$$ 
\end{lemma}
\begin{proof}
To achieve the chain conditions above, it suffices to show that  at most $a-1$ of the $d_1, \dots,  d_{\mu}$ and at most $a-1$ of the $h_1, \dots, h_{\nu}$  are greater than zero. To do this, the arguments in the previous two lemmas show that it is enough to prove that $c_1 \leq a-1$ and $g_1 \leq a-1.$

Assume, for the sake of contradiction, that at least one of $c_1$ or $g_1$ is strictly greater than $a-1$. 
We've already shown in \cref{rightlatticelemma} that $c_1 = g_1$, and that both of these values are bounded by $a$. So, we must consider the case where $c_1 = g_1 = a$, in the hope of finding a contradiction.

We show inductively that $d_i = h_i = 0$ for all $i \geq 0$. This is true for $i=0$ by the diagrams in \cref{fig:Tidiag,fig:Sidiag}, since we've assumed $c_1=g_1=a=g-a$. 

Now we show that $h_1 = 0$.  Because $d_0 = 0,$
$$T_0 \cap \tau(T_0) = \tau(T_{-1}).$$
Using the definition of $T_0$ and $T_{-1},$
$$M_0 \cap \tau(M_0) = pM_0^\vee.$$
Taking the dual of both sides yields
$$M_0^\vee + \tau(M_0^\vee)  = p^{-1} M_0^{\vee \vee}.$$
Since, by \cref{prop:0and1conditions} taking the dual twice is the same as applying $\tau$, 
$$M_0^\vee + \tau(M_0^\vee)  = p^{-1} \tau(M_0).$$
Now we apply the operator $pF^{-1}$ to both sides and use the fact that $pF^{-1}(M_0^\vee) = M_1$,
$$M_1 + \tau(M_1) = F^{-1}( \tau(M_0) ),$$
which is simply notation for the equality
$$S_{1} = F^{-1}(\tau(T_{0})).$$
Then, again using the fact that $d_0 = 0,$
\begin{align*}
S_1 \cap \tau(S_1) = F^{-1}(\tau(T_{0}) \cap \tau^2(T_0)) = F^{-1}(\tau^2(T_{-1}))
= \tau F^{-1}( p M_0^\vee) = \tau(M_1) = \tau(S_0).
\end{align*}
Hence, $h_1=0.$ The parallel argument beginning with the fact that $S_0 \cap \tau(S_0) = \tau(S_{-1})$ shows that $d_1=0.$

Now, let $j \geq 1$ and assume that $d_i = h_i = 0$ for all $0 \leq i \leq j$. Because $d_0=0$,
$$T_0 \cap \tau(T_0) \cap \cdots \cap \tau^{j+1} (T_0) = \tau(T_{-1} ) \cap \tau^2(T_{-1}) \cap \cdots \cap \tau^{j+1}(T_{-1}).$$ 
Then, taking the dual and applying the operator $pF^{-1}$ to both sides, 
$$M_1 + \tau(M_1) + \cdots + \tau^{j+1} (M_1) = F^{-1}( \tau(M_0) + \tau^2(M_0) + \cdots + \tau^{j+1}(M_0) ),$$
which is simply notation for the equality
$$S_{j+1} = F^{-1}(\tau(T_{j})).$$
The same argument applied to $j-1$ shows that
$$S_{j} = F^{-1}(\tau(T_{j-1})).$$

Now, using our assumption that $d_{j}=0$,
\begin{align*}
S_{j+1} \cap \tau(S_{j+1}) = F^{-1}(\tau(T_{j}) \cap \tau^2(T_{j}) ) = F^{-1}( \tau^2(T_{j-1})) = \tau(S_j).
\end{align*}

And so $h_{j+1} = 0$. The parallel argument, with the roles of $M_0$ and $M_1$ reversed, shows that $d_{j+1} = 0$.
Inductively, $d_i = 0$ and $h_i = 0$ for all $i$, which is a contradiction to the existence of $n$ and $m$ in \cref{lem:existenceofmandn}.
\end{proof}

Combining the previous three lemmas produces the following proposition.

\begin{proposition}\label{prop:lambda}
Let $M$ be a supersingular unitary Dieudonn\'e module of signature $(a,g-a)$ over an algebraically closed field $\kk$. Define
$$\Lambda \coloneqq M + \tau(M) + \tau^2(M) + \cdots.$$
Then $\Lambda$ is the Dieudonn\'e module of a superspecial p-divisible group over $\kk$, and
$$p^a \Lambda \subseteq M \subseteq \Lambda.$$

Further, when $g = 2a$,
$$p^{a-1} \Lambda \subseteq M \subseteq \Lambda.$$
\end{proposition}

\begin{proof}
Note that $\Lambda = \Lambda_0 + \Lambda_1$, where 
\begin{align*}
    \Lambda_0 &= M_0 + \tau(M_0) + \tau^2(M_0) + \cdots = T_\mu,\\
    \Lambda_1 &= M_1 + \tau(M_1) + \tau^2(M_1) + \cdots = S_\nu,
\end{align*}
and then apply \cref{leftlatticelemma,rightlatticelemma} to see that
$$p^a \Lambda \subseteq M \subseteq \Lambda.$$

Because the $F$ and $V$ operators commute with $\tau$ and $M$ is stable under $F$ and $V$, $\Lambda$ is also stable under $F$ and $V$. So, $\Lambda$ is the Dieudonn\'e module of a $p$-divisible group. This $p$-divisible group is superspecial, because the condition that $\tau(\Lambda) = \Lambda$ is equivalent to $F \Lambda = V \Lambda$.

When $g = 2a$, we have $p^{a-1} \Lambda \subseteq M \subseteq \Lambda$ by \cref{parallel}. \end{proof}

\begin{corollary}\label{cor:isog}
   Let $X$ be a supersingular unitary p-divisible group of signature $(a,g-a)$. When $g > 2a$ there exists a superspecial  $p$-divisible group $\uxx$ and an isogeny $\rho: X \rightarrow  \uxx$ with $\mathrm{Ker}(\rho) \subseteq X[p^a]$.  When $g = 2a$, there exists an isogeny $\rho: X \rightarrow \uxx$ with $\mathrm{Ker}(\rho) \subseteq X[p^{a-1}]$.
\end{corollary}

\begin{proof}
This follows from \cref{prop:lambda} and covariant Dieudonn\'e theory.
\end{proof}

\section{Minimal Heights}\label{sec:heights}

The goal of this section is to translate the results of the previous section into the language of \emph{minimal heights}, and to give a complete description of minimal heights occurring for supersingular unitary $p$-divisible groups of signature $(a,g-a)$. 

\begin{definition}[cf. {\cite[Definitions~2.5, 2.7]{TraversoII}}]\label{def:minimal height}
Let $X$ be a supersingular $p$-divisible group. The \emph{minimal height} of $X$ is the minimal integer $q_X$ such that there exists an isogeny $\rho: X \rightarrow \uxx$ to a superspecial $p$-divisible group, with $\ker(\rho) \subseteq X[p^{q_X}]$.

As noted in \cite{TraversoII}, the minimal height $q_X \geq 0$ can also be thought of as a ``minimal distance" from $X$ to a superspecial $p$-divisible group; in particular, $q_X = 0$ exactly when $X$ itself is a superspecial $p$-divisible group (cf. \cite[Lemma~2.6]{TraversoII}).
\end{definition}

\begin{remark}
It is natural to ask: If we are only concerned with $p$-divisible groups $X$ with unitary structure of signature $(a,g-a)$, why not also require $\uxx$ to have the same structure, and for $\rho$ to respect both the action and the quasi-polarization? 

In fact, because the $p$-divisible group of a supersingular elliptic curve admits action and quasi-polarization of degenerate signature $(1,0)$ (this uses the fact that $K$ is the unramified degree two extension of $\qq_p$), any superspecial $\uxx$ can be endowed with action and quasi-polarization of any signature $(a,g-a),$ as long as $g = \mathrm{dim}(\uxx).$ However, we don't require $\rho$ to respect this additional structure, especially the quasi-polarization, for two reasons: first, experience with the use of lemmas similar \cref{leftlatticelemma} in the study of the geometry of Rapoport-Zink spaces tells us that our formulation is more likely to be helpful. Second, our definition is the one necessary for the applications to isogeny cutoffs in \cref{sec:isogenythm}.

\end{remark}

\begin{lemma}\label{lem:minht}
Let $X$ be a supersingular $p$-divisible group over $\kk$, with associated Dieudonn\'e module $M$, and let $\Lambda = M + \tau(M) + \tau^2(M) + \cdots $. Then its minimal height $q_X$ is the minimal integer $r_M$ such that
$$p^{r_M} \Lambda \subseteq M \subseteq \Lambda.$$
\end{lemma}

\begin{proof}
By covariant Dieudonn\'e theory, the inclusions $p^{r_M} \Lambda \subseteq M \subseteq \Lambda$ define an isogeny $\rho_{\Lambda}: X \rightarrow X_{\Lambda}$ to the superspecial $p$-divisible group $X_{\Lambda}$, with the property that $\ker(\rho_{\Lambda}) \subseteq X[p^{r_M}]$.  So, $q_X \leq r_M$. 

On the other hand, by the definition of $q_X$, there exists a superspecial $p$-divisible group $\uxx$ and an isogeny $\rho: X \rightarrow \uxx$ with $\ker(\rho) \subseteq X[p^{q_X}]$. If $M_{\uxx}$ is the $p$-adic Dieudonn\'e module of $\uxx$, this gives an inclusion
$$M \subseteq M_{\uxx}$$
with the property that $p^{q_X} M_{\uxx} \subseteq M \subseteq M_{\uxx}$. 

Since $\uxx$ is superspecial, $F M_{\uxx} = V M_{\uxx}$ and so $\tau(M_{\uxx}) = M_{\uxx}$. But then,
$$M \subseteq M  + \tau(M) + \tau^2(M) + \cdots \subseteq M_{\uxx},$$
which is simply the notation for the inclusions $M \subseteq \Lambda \subseteq M_{\uxx}$. In particular, \[p^{q_X} \Lambda \subseteq p^{q_X} M_{\uxx} \subseteq M \subseteq \Lambda \subseteq M_{\uxx}.\] By the minimality of $r_M$, we have $q_X \geq r_M$. Thus, $q_X=r_M$. 
\end{proof}

\begin{lemma}\label{lem:max-prod-min-ht}
    Let $X$ and $X'$ be supersingular $p$-divisible groups over $\kk$, of minimal height $q_X$ and $q_{X'}$ respectively. Then the minimal height of $X \times_\kk X'$ is $\mathrm{max}(q_X, q_{X'})$.
\end{lemma}
\begin{proof}
    Observe that since $X$ and $X'$ are supersingular $p$-divisible groups over $\kk$, so is $X \times_\kk X'$. We can therefore apply \cref{lem:minht} to $X \times_\kk X'$. 
    
    Let $M$ and $M'$ be the Dieudonn\'e module associated to $X$ and $X'$ respectively, and by covariant Dieudonn\'e theory we may take the Dieudonn\'e module associated to $X \times_\kk X'$ to be $M \oplus M'$. Define $\Lambda = M + \tau(M) + \tau^2(M) + \cdots $ and $\Lambda' = M' + \tau(M') + \tau^2(M') + \cdots $ and note that we have
    \[(M \oplus M') + \tau(M \oplus M') + \tau^2(M \oplus M') + \cdots  = \Lambda \oplus \Lambda'.\]
    Therefore, by \cref{lem:minht}, the minimal height of $X \times_\kk X'$ is the minimal integer $r$ such that
    \[p^r (\Lambda \oplus \Lambda') \subseteq M \oplus M' \subseteq \Lambda \oplus \Lambda'\tag{$*$}\label{eq:min-ht-prod}\]

    We first observe that $r = \max(q_X, q_{X'})$ satisfies \cref{eq:min-ht-prod}:
    \[p^{\max(q_X, q_{X'})} \Lambda \oplus p^{\max(q_X, q_{X'})}\Lambda' \subseteq p^{q_X}\Lambda \oplus p^{q_{X'}}\Lambda' \subseteq M \oplus M'\]
    where the latter is a consequence of the characterization of $q_X$ and $q_{X'}$ given in \cref{lem:minht}.
    
    Now, let $r$ be any integer satisfying \cref{eq:min-ht-prod}. We note that $p^r\Lambda \subseteq M$ and $p^r\Lambda' \subseteq M'$. It follows from the minimality statement in \cref{lem:minht} applied to $X$ and $X'$ that $r \geq q_X$ and $r \geq q_{X'}$ respectively; equivalently, $r \geq \max(q_X, q_{X'})$.

    Hence, applying the minimality statement in \cref{lem:minht} to $X \times_\kk X'$ lets us conclude that the minimal height of $X \times_\kk X'$ is $\mathrm{max}(q_X, q_{X'})$.
\end{proof}

We now give a series of examples of specific unitary $p$-divisible groups and their minimal heights. These examples are used in the proof of \cref{thm:allabtminht}.

\begin{example}\label{ex:oddMab}
Assume $g$ is odd. In this example, we construct a supersingular unitary $p$-divisible group of signature $(a,g-a)$ over $\kk$,  with minimal height $a$. Let $r = (g-1)/2$; then our standard assumption $g \geq 2a$ on is equivalent to $a \leq r$.

We define a unitary Dieudonn\'e module $M_{(a,g-a)} = (M_{(a,g-a)}, F, V,  \langle  \cdot ,  \cdot \rangle, M_{(a,g-a)} = M_0 \oplus M_1 ) $. Set $M_0 = \linspan_{W(\kk) }  \{ e_1, \dots, e_{g} \} $ and $M_1 = \linspan_{W(\kk) } \{  f_1, \dots, f_{g} \}$, and define $M_{(a,g-a)} \coloneqq M_0 \oplus M_1$. 

Let $F$ be the $\sigma$-linear extension  of the operator $e_i \mapsto p^{\alpha_i} f_i$ and $f_j \mapsto p^{\beta_j} e_{j+1}$, where
\begin{align*}
\alpha_i =\begin{cases}
        0 &\text{if } 1 \leq i \leq a, \\
        1 &\text{if } a + 1 \leq i \leq g,
    \end{cases} \hspace{2cm}
\beta_j = \begin{cases}
        1 &\text{if } r+1 \leq j \leq r+ a, \\
        0 &\text{if $j < r + 1$ or $j > r + a$,} 
    \end{cases}
\end{align*}
and let $V \coloneqq pF^{-1}$. Observe that we have $\alpha_i + \beta_{r + i} = 1$ for any $i$ modulo $g$.

Note that, by construction, $pM_0 \subseteq_{g-a} FM_1 \subseteq_a M_0$ and $pM_1 \subseteq_a FM_0 \subseteq_{g-a} M_1$.

\vspace*{-0.5\baselineskip}
{
\begin{center}
\begin{figure}[h!]
\begin{adjustbox}{max width=.5\textwidth}
\begin{tikzpicture}

  \fill(-5.76,-2.27) circle (2pt);
  \node at (-5.76,-2.27) (e1) {};
  \node[anchor=east] at (-5.76,-2.27) {$e_1$};

 \fill(-6.55,-1.42)circle (2pt); 
 \node at (-6.55,-1.42) (f1) {}; 
 \node[anchor=east] at (-6.55,-1.42) {$f_{1}$};

 \fill(-6.95,-0.48) circle (2pt); 
 \node at (-6.95,-0.48) (e2) {}; 
 \node[anchor=east] at (-6.95,-0.48) {$e_{2}$};

 \fill(-6.95,0.48) circle (2pt);
 \node at (-6.95,0.48) (f2) {};
 \node[anchor=east] at (-6.95,0.48) {$f_{2}$};

 \fill(-6.55,1.42) circle (2pt); 
 \node at (-6.55,1.42) (ea) {}; 
 \node[anchor=south east] at (-6.55,1.42) {$e_{a}$};

 \fill(-5.76,2.27) circle (2pt); 
 \node at (-5.76,2.27) (fa) {}; 
 \node[anchor=south east] at (-5.76,2.27) {$f_{a}$};

 \fill(-4.64,2.99) circle (2pt);
 \node at (-4.64,2.99) (ea1) {};
 \node[anchor=south east] at (-4.64,2.99) {$e_{a+1}$};

 \fill(-3.25,3.54) circle (2pt); 
 \node at (-3.25,3.54) (fa1) {}; 
 \node[anchor=south east] at (-3.25,3.54) {$f_{a+1}$};

 \fill(-1.67,3.88) circle (2pt); 
 \node at (-1.67,3.88) (ea2) {}; 
 \node[anchor=south] at (-1.67,3.88) {$e_{a+2}$};

 \fill(0,4) circle (2pt); 
  \node at (0,4) (fa2) {}; 
 \node[anchor=south] at (0,4) {$f_{a+2}$};

 \fill(1.67,3.88) circle (2pt); 
  \node at (1.67,3.88) (ea3) {}; 
 \node[anchor=south] at (1.67,3.88) {$e_{a+3}$};

 \fill(3.25,3.54) circle (2pt); 
  \node at (3.25,3.54) (er) {}; 
 \node[anchor=south west] at (3.25,3.54) {$e_{r}$};

 \fill(4.64,2.99) circle (2pt); 
 \node at (4.64,2.99) (fr) {}; 
 \node[anchor=south west] at (4.64,2.99) {$f_{r}$};

 \fill(5.76,2.27) circle (2pt); 
  \node at (5.76,2.27) (er1) {}; 
 \node[anchor=south west] at (5.76,2.27) {$e_{r+1}$};

  \fill(6.55,1.42) circle (2pt);
  \node at (6.55,1.42) (fr1) {};
  \node[anchor=west] at (6.55,1.42) {$f_{r+1}$};

 \fill(6.95,0.48) circle (2pt); 
 \node at (6.95,0.48) (er2) {}; 
 \node[anchor=west] at (6.95,0.48) {$e_{r+2}$};

 \fill(6.95,-0.48) circle (2pt); 
 \node at (6.95,-0.48) (fr2) {}; 
 \node[anchor=west] at (6.95,-0.48) {$f_{r+2}$};

 \fill(6.55,-1.42)  circle (2pt); 
 \node at (6.55,-1.42)  (era1) {}; 
 \node[anchor=west] at (6.55,-1.42)  {$e_{r+a+1}$};

 \fill(5.76,-2.27) circle (2pt); 
 \node at (5.76,-2.27) (fra1) {}; 
 \node[anchor=north west] at (5.76,-2.27) {$f_{r+a+1}$};

 \fill(4.64,-2.99) circle (2pt); 
 \node at (4.64,-2.99) (era2) {}; 
 \node[anchor=north west] at (4.64,-2.99) {$e_{r+a+2}$};

 \fill(3.25,-3.54) circle (2pt);
 \node at (3.25,-3.54) (fra2) {}; 
 \node[anchor=north west] at (3.25,-3.54) {$f_{r+a+2}$};

 \fill(1.67,-3.88) circle (2pt);
 \node at (1.67,-3.88) (era3) {};
 \node[anchor=north west] at (1.67,-3.88) {$e_{r+a+3}$};

 \fill(0,-4) circle (2pt); 
 \node at (0,-4) (fra3) {}; 
 \node[anchor=north] at (0,-4) {$f_{r+a+3}$};

 \fill(-1.67,-3.88) circle (2pt); 
 \node at (-1.67,-3.88) (era4) {}; 
 \node[anchor=north] at (-1.67,-3.88) {$e_{r+a+4}$};

 \fill(-3.25,-3.54) circle (2pt); 
  \node at (-3.25,-3.54) (eg) {};
 \node[anchor=north]  at (-3.25,-3.54) {$e_{g}$};

 \fill(-4.64,-2.99) circle (2pt); 
 \node at (-4.64,-2.99) (fg) {};
 \node[anchor=north] at (-4.64,-2.99) {$f_{g}$};

 \path[->]
        (eg) edge[bend left=5]  node[anchor= west] {}   node[anchor= south west] {$p$} (fg);

 \path[->]
        (fg) edge[bend left=5]  node[anchor= west] {}  (e1);

 \path[->]
        (e1) edge[bend left=10]  node[anchor= west] {}  (f1);

 \path[->]
        (f1) edge[bend left=10]  node[anchor= north west] {}  (e2);

 \path[->]
        (e2) edge[bend left=10]  node[anchor=north  west] {}  (f2);

 \path[-]
        (f2) edge[dotted, bend left=10]   (ea);

 \path[->]
        (ea) edge[bend left=10]  node[anchor= north] {}  (fa);

 \path[->]
        (fa) edge[bend left=5]  node[anchor= north] {}  (ea1);

 \path[->]
        (ea1) edge[bend left=5]  node[anchor= north] {}  node[anchor= north] {$p$}  (fa1);

 \path[->]
        (fa1) edge[bend left=5]  node[anchor= north] {}  (ea2);

 \path[->]
        (ea2) edge[bend left=2]  node[anchor= north] {} node[anchor= north] {$p$}  (fa2);

 \path[->]
        (fa2) edge[bend left=2]  node[anchor= north] {}  (ea3);

 \path[-]
        (ea3) edge[dotted, bend left=5]   (er);

 \path[->]
        (er) edge[bend left=5]  node[anchor= north] {} node[anchor= north east] {$p$}  (fr);

 \path[->]
        (fr) edge[bend left=5]  node[anchor= north east] {}  (er1);

 \path[->]
        (er1) edge[bend left=10]  node[anchor= east] {} node[anchor= north east] {$p$} (fr1);

 \path[->]
        (fr1) edge[bend left=10]  node[anchor=south east] {}  node[anchor=east] {$p$} (er2);
        
 \path[->]
        (er2) edge[bend left=10]  node[anchor=south east] {}  node[anchor=east] {$p$} (fr2);

 \path[-]
        (fr2) edge[dotted, bend left=10]   (era1);

 \path[->]
        (era1) edge[bend left=10]  node[anchor=south] {}   node[anchor=south east] {$p$}  (fra1);

 \path[->]
        (fra1) edge[bend left=5]  node[anchor=south] {}  (era2);

 \path[->]
        (era2) edge[bend left=5]  node[anchor=south] {} node[anchor=south] {$p$} (fra2);

 \path[->]
        (fra2) edge[bend left=5]  node[anchor=south] {}  (era3);

 \path[->]
      (era3) edge[bend left=2]  node[anchor=south] {$p$}     (fra3);

\path[->]
      (fra3) edge[bend left=2]  node[anchor=south] {}     (era4);

 \path[-]
        (era4) edge[dotted, bend left=5]   (eg);

\end{tikzpicture}
\end{adjustbox}
\setlength{\belowcaptionskip}{-2\baselineskip}
\caption{$F$ Operator}
  \label{fig:Fdiag}
\end{figure} 
\end{center}}

Define the pairing
$\langle \cdot , \cdot \rangle: M_{(a,g-a)} \times M_{(a,g-a)} \rightarrow W(\kk)$ by choosing $ \delta \in W(\kk)^\times$ such that $\delta^\sigma = - \delta$, and setting
$$\langle e_i, f_{r+i} \rangle = \delta,\quad 
\text{and $\langle e_i, f_j \rangle = 0$ otherwise,}$$ declaring $M_0$ and $M_1$ to both be totally isotropic with respect to $\langle \cdot , \cdot \rangle$, and setting $\langle f_j, e_i \rangle = - \langle e_i, f_j \rangle$. All indices are considered modulo $g$ (with $0 = g$).

To show that the pairing has the property $\langle F(x), y \rangle = \langle x, V(y) \rangle^\sigma$ for $x, y \in M_{a,(g-a)}$, it suffices to show this for when $x,y \in \{e_i,f_j\ :\ 1\leq i,j \leq g\}$ by bilinearity of the pairing. Since $V = pF^{-1}$, replacing $y$ with $F(y)$ and using bilinearity of the pairing, it further suffices to show that $\langle F(x), F(y) \rangle = p\langle x, y \rangle^\sigma$. Finally, since the pairing is defined to be isotropic on $M_0$ and $M_1$, we may assume $x = e_i$ and $y = f_j$ for some $i,j$.

The RHS $p\langle e_i,f_j \rangle^\sigma \neq 0$ exactly when $j \equiv r + i \bmod{g}$ by definition of the pairing. Now, the LHS
\[\langle F(e_i),F(f_j)\rangle = p^{\alpha_i + \beta_j}\langle f_i,e_{j+1}\rangle = -p^{\alpha_i + \beta_j}\langle e_{j+1},f_i\rangle\]
is non-zero exactly when $i \equiv r + (j + 1) \bmod{g}$, or equivalently, $j \equiv i-(r+1) \equiv i + r\bmod{g}$, since $g = 2r + 1$. Hence, both LHS and RHS vanish (and therefore are equal) unless $j \equiv r + i \bmod{g}$. 

In the case $j \equiv r + i \bmod{g}$, we have
\[\langle F(e_i),F(f_{r+i})\rangle = -p^{\alpha_i + \beta_{r+i}}\langle e_{r+i+1},f_i\rangle = p(-\delta) = p\delta^\sigma = p\langle e_i,f_{r+i}\rangle,\]
since $\alpha_i + \beta_{r+i} = 1$. Thus, the pairing indeed has the property $\langle F(x), y \rangle = \langle x, V(y) \rangle^\sigma$.

To show that $M_{(a,g-a)}$ is supersingular, note that $F^{2g}(v) = p^g(v)$ on all basis vectors $v$ and therefore $\Phi = p^{-g}F^{2g}$ is a $\sigma^{2g}$-linear automorphism of $M_{(a,g-a)}$ with all Newton polygon slopes equal to $0$. Since any Newton polygon slope $\lambda$ of $(M_{(a,g-a)},F)$ contributes the slope $2g \lambda - g$ to $(M_{(a,g-a)},\Phi)$, we conclude that all Newton polygon slopes of $M_{(a,g-a)}$ are $\frac{g}{2g} = \frac{1}{2}$. Thus, $M_{(a,g-a)}$ is a supersingular unitary Dieudonn\'e module of signature $(a,g-a)$.

Let $X_{(a,g-a)}$ be the corresponding $p$-divisible group. We now compute its minimal height. 

Let $\Lambda_{(a,g-a)} \coloneqq M_{(a,g-a)} + \tau(M_{(a,g-a)}) + \tau^2 M_{(a,g-a)} + \cdots$. By definition of $F$ and $V$ on $M_{(a,g-a)}$, 
$$\tau^a(e_1) = p^{-a} e_{a+1},$$
and so while $p^a \Lambda \subseteq M_{(a,g-a)} \subseteq \Lambda$, it is not true that $p^{a-1} \Lambda \subseteq M_{(a,g-a)}$. Then, by \cref{lem:minht}, the minimal height of $X_{(a,g-a)}$ is $a$. 

We observe, and later use, the fact that the construction and properties of $M_{(a,g-a)}$ make sense even if we assume the degenerate case of $a = 0$, and we will use this construction to refer to $M_{(0,1)}$ below.

\end{example}

\begin{remark}
Though not necessary for the results in this paper, the reader may find the reasoning behind the construction in \cref{ex:oddMab} interesting: as we observe in \cref{thm:allabtminht}, this is the largest possible minimal height. We know from \cite[Theorem 1.1]{TraversoII} that ``having largest possible minimal height" is an open condition, so to create our example we needed to find a suitably generic $p$-divisible group. We found such a candidate by choosing a supersingular $p$-divisible group $X$ with the property that the Ekedahl--Oort stratum defined by $X[p]$ in the relevant unitary Shimura variety has dimension larger than the supersingular locus, so this $X$ was likely to be generic in other aspects as well.
\end{remark}

\begin{example}\label{ex:evenMab}
Assume $g$ is even and $g\neq 2a$. In this example, we construct a supersingular unitary $p$-divisible group of signature $(a,g-a)$ over $\kk$,  with a minimal height of $a$.

First note that in this setting, it is also true that $g > 2a-1$. Thus, we can consider $M_{(a,g-a-1)}$ (with its additional structure) and define $M_{(0,1)}$ as in \cref{ex:oddMab}. We then define $M_{(a,g-a)}$ as
$$M_{(a,g-a)} = M_{(a,g-a-1)} \oplus M_{(0,1)}$$
with the $F$ and $V$ maps and the quasi-polarization defined as a product as well.

Note that $M_{(a,g-a)}$ is a unitary $p$-adic Dieudonn\'e module of signature $(a,g-a)$ because it has been constructed as a product of $p$-adic Dieudonn\'e modules of signatures $(a,g-a-1)$ and $(0,1)$. Since $M_{(a,g-a)}$  is a product of supersingular Dieudonn\'e modules, it is also supersingular.

Let $X_{(a,g-a)}$ be the $p$-divisible group associated to $M_{(a,g-a)}$. By construction,
$$X_{(a,g-a)} \cong X_{(a,g-a-1)} \times_{\kk} X_{(0,1)}.$$
Note that, by the previous example, $X_{(a,g-a-1)}$ has minimal height $a$, since $a< g-a-1$. Since $F M_{(0,1)} = V M_{(0,1)}$, the $p$-divisible group $X_{(0,1)}$ is superspecial, so has minimal height $0$. Then, $X_{(a,g-a)}$ has minimal height $\mathrm{max}(a,0) = a$, by \cref{lem:max-prod-min-ht}.

\end{example}

\begin{example}\label{ex:Maa}
In this example, we construct a supersingular unitary $p$-divisible group  of signature $(a,a)$ over $\kk$,  with a minimal height of $a-1$.

Let $ M_{(a-1,a)}$ be the $p$-adic Dieudonn\'e module of signature $(a-1,a)$ constructed in \cref{ex:oddMab}. Define $M_{(1,0)}$ exactly as $M_{(0,1)}$ of \cref{ex:oddMab}, but with the modification $(M_{(1,0)})_{i} \coloneqq (M_{(0,1)})_{1-i}$ for $i \in \{0,1\}$. As a result, $M_{(1,0)}$ is a superspecial $p$-adic unitary Dieudonn\'e module of (degenerate) signature $(1,0)$. 

Now, define $M_{(a,a)}$ as the product
$$M_{(a,a)} \coloneqq M_{(a-1,a)} \oplus M_{(1,0)},$$
equipped with product $F$ and $V$ operators and product quasi-polarization. By construction, $M_{(a,a)}$ is a supersingular unitary Dieudonn\'e module of signature $(a,a)$. Let $X_{(a,a)} \cong  X_{(a-1,a)} \times_{\kk} X_{(1,0)} $ be the corresponding $p$-divisible group. Since $ X_{(a-1,a)}$ has minimal height $a-1$ and $X_{(1,0)}$ has minimal height $0$, it follows that $ X_{(a,a)} $ has minimal height $a-1$, by \cref{lem:max-prod-min-ht}.

\end{example}

The three previous examples combine to construct a supersingular $p$-divisible group $X_{(a,g-a)}$ of signature $(a,g-a)$, for any signature. 
When $g > 2a$, this $p$-divisible group has minimal height $a$. When $g = 2a$, the minimal height is $a-1$.

We are now prepared to prove \cref{thm:allabtminht}.
    
\begin{proof}[Proof of \cref{thm:allabtminht}]

First consider the case where $g > 2a.$ 
By the first statement in \cref{cor:isog}, the minimal height of any supersingular unitary $p$-divisible group $X$ over $\kk$ of signature $(a,g-a)$ is at most $a$.

By \cref{ex:oddMab,ex:evenMab} (depending on the parity of $g$), there exists a supersingular unitary $p$-divisible group $X_{(a,g-a)}$ of signature $(a,g-a)$, with minimal height $a$. Extending these examples, for any $0 \leq q \leq a$,
$$X_{(q,g-a)} \times_\kk \underbrace{X_{(1,0)} \times_\kk X_{(1,0)} \times_\kk \cdots \times_\kk X_{(1,0)}}_{\text{$a-q$ times}}$$
is a supersingular unitary  $p$-divisible group that has a minimal height of $q$.

Now consider the case where $g = 2a$. By the second statement in \cref{cor:isog}, the minimal height of any supersingular unitary $p$-divisible group $X$ over $\kk$ of signature $(a,a)$ is at most $a-1$.

In \cref{ex:Maa}, we constructed a supersingular unitary $p$-divisible group $X_{(a,a)}$ of signature $(a,a)$, with minimal height $a-1$. Extending this example, for any $0 \leq q \leq a-1$
$$X_{(q,a)} \times_\kk \underbrace{X_{(1,0)} \times_\kk X_{(1,0)} \times_\kk \cdots \times_\kk X_{(1,0)}}_{\text{$a-q$ times}}$$
is a supersingular unitary $p$-divisible group of  minimal height $q$.
\end{proof}

\section{Isogeny Cutoffs}\label{sec:isogenythm}

In this section, we apply the results of the previous section to understand \emph{isogeny cutoffs} of supersingular unitary $p$-divisible groups. Throughout, $a$ and $g$ are as in \cref{not:a-n-g}.

\begin{definition}[cf. {\cite[Definition~2.3]{TraversoII}}]\label{def:isogney cutoff}
The \emph{isogeny cutoff} of a $p$-divisible group $X$ over $\kk$ is the minimal integer $b_X$ such that, for any other $p$-divisible group $X'$, an isomorphism $X[p^{b_X}] \cong X'[p^{b_X}]$ implies the existence of an isogeny $X \rightarrow X'$.

In other words, $b_X$ is the smallest level of $p$-power torsion that determines $X$ up to isogeny; and, as noted in \cite{TraversoII}, one has $b_X \geq 1$. In particular, $b_{\uxx} = 1$, where $\uxx$ is a superspecial $p$-divisible group.
\end{definition}

\begin{definition}\label{def: supsing-isog-bound}
We define the \emph{signature $(a,g-a)$ supersingular isogeny bound} to be the least $B_{(a,g-a)}$ such that $b_X \leq B_{(a,g-a)}$ for all  supersingular unitary $p$-divisible groups $X$ of signature $(a,g-a)$, defined over any algebraically closed field $\kk$ of characteristic $p$. Necessarily $1 \leq B_{(a,g-a)}$.
\end{definition}

\begin{definition}
If $X$ is a unitary $p$-divisible group of fixed signature $(a,g-a)$, define the \emph{enhanced isogeny cutoff} $\widetilde{b}_X$, as the minimal integer such that, for any other unitary $p$-divisible group $X'$ also of signature $(a,g-a)$, an isomorphism $X[p^{\widetilde{b}_X}] \cong X'[p^{\widetilde{b}_X}]$ \emph{respecting the action and quasi-polarization on both $X$ and $X'$}, implies the existence of an isogeny $X \rightarrow X'$. Note that then $ \widetilde{b}_X \leq b_X$. 

Further, note that it follows from our convention $0 < a$ that we must have $\widetilde{b}_X \geq 1$ (this is because a relevant unitary Shimura variety of signature $(a,g-a)$ has at least one Newton stratum besides the supersingular locus). 

Similar to the supersingular isogeny bound in \cref{def: supsing-isog-bound} define the \emph{enhanced signature $(a,g-a)$ supersingular isogeny bound} $\widetilde{B}_{(a,g-a)}$. We have $1\leq \widetilde{B}_{(a,g-a)} \leq B_{(a,g-a)}$. 

However, note that by \cite[Lemma 5.1]{VollaardWedhorn} the values of $\widetilde{b}_X$ and $\widetilde{B}_{(a,g-a)}$ do not depend on whether or not we require the isogeny $X \rightarrow X'$ to respect the unitary action and quasi-polarization.
\end{definition}

Recall the unitary $p$-divisible group $X_{(a,g-a)}$ and its associated Dieudonn\'e module $M_{(a,g-a)}$ from \cref{ex:oddMab}. We perturb this construction to construct two explicit examples of $p$-divisible groups.   The former (\cref{ex:Mk}) helps us control $\widetilde{B}_{(a,g-a)}$ and the latter (\cref{ex:Mk-notuni}) helps control $B_{(a,g-a)}$.

\begin{example}\label{ex:Mk}
Let $t \geq 1$, and consider $(a,g-a)$ with $g$ odd. In what follows, we construct a $p$-divisible group $\prescript{t}{}{X_{(a,g-a)}}$ such that $\prescript{t}{}{X_{(a,g-a)}}[p^t] \cong X_{(a,g-a)}[p^t]$ (where $X_{(a,g-a)}$ is the $p$-divisible group constructed from $(M_{(a,g-a)}, F, V, \langle \cdot, \cdot \rangle, M_{(a,g-a)} = M_0 \oplus M_1)$ 
  in \cref{ex:oddMab}). As before, $r$ is such that $g = 2r + 1$, and $g > 2a > 0$ is equivalent to $1 \leq a \leq r$.

Let $\prescript{t}{}{M_{(a,g-a)}}  = (M_{(a,g-a)}, F_t, V_t, \langle \cdot, \cdot \rangle, M_{(a,g-a)} = M_0 \oplus M_1)$, so all the data of $\prescript{t}{}{M_{(a,g-a)}}$ agrees with that of $M_{(a,g-a)}$ in \cref{ex:oddMab}, except for $F_t$ and $V_t$, which we obtain by perturbing $F$ and $V$. Define $F_t \coloneqq F + \zeta_t$ and $V_t \coloneqq V + \xi_t$, where $\zeta_t$ and $\xi_t$ are $0$ on all basis elements except for the following:
\begin{align*}
    \zeta_t(f_a), &= p^t e_1 & \xi_t(e_{a+1}) &= -p^{t+1}f_g,\\
    \zeta_t(e_{r+1}) &= -p^{t+1}f_{r+a+1}, & \xi_t(f_{r+1}) &= p^t e_{r+a+1}
\end{align*}

By construction, the action on $\prescript{t}{}{M_{(a,g-a)}}$ is still of signature $(a,g-a)$, we have $F_t \circ V_t = V_t \circ F_t = p$, and $F \equiv F_t \mod p^t$ and $V \equiv V_t \mod p^t$. 

To check the property that $\langle F_t(x), y \rangle = \langle x, V_t(y) \rangle^\sigma$ it suffices to check $\langle \zeta_t(x), y \rangle = \langle x, \xi_t(y) \rangle^\sigma$ since we have already verified that $F$ and $V$ satisfy this property in \cref{ex:oddMab}. It suffices to check this on basis elements, and, by definition, only two pairs give non-zero outputs: $(x,y) = (f_a,f_{r+1})$ or $(e_{r+1},e_{a+1})$.
\begin{align*}
    \langle \zeta_t(f_a), f_{r+1} \rangle &= p^t\langle e_1,f_{r+1}\rangle = p^t\delta, \\ \langle f_a, \xi_t(f_{r+1}) \rangle^\sigma &= -p^t\langle e_{r+a+1},f_a \rangle^\sigma = p^t(-\delta^\sigma) = p^t\delta;
\end{align*}
and
\begin{align*}
    \langle \zeta_t(e_{r+1}), e_{a+1} \rangle &= -p^{t+1}\langle f_{r+a+1},e_{a+1}\rangle = p^{t+1}\delta, \\ \langle e_{r+1}, \xi_t(e_{a+1}) \rangle^\sigma &= -p^{t+1}\langle e_{r+1},f_g \rangle^\sigma = p^{t+1}(-\delta^\sigma) = p^{t+1}\delta.
\end{align*}

Therefore, for each $t \geq 1$, we've constructed a unitary $p$-divisible group $\prescript{t}{}{X_{(a,g-a)}}$ of signature $(a,g-a)$, such that $\prescript{t}{}{X_{(a,g-a)}}[p^t] \cong X_{(a,g-a)}[p^t]$. In fact, this isomorphism respects the induced action and quasi-polarization on both $\prescript{t}{}{X_{(a,g-a)}}[p^t]$ and $X_{(a,g-a)}[p^t]$, since by construction the splitting of Dieudonn\'e modules into $0$ and $1$ parts was identical, and the quasi-polarization form on both Dieudonn\'e modules was also identical.
\end{example}

\begin{lemma}\label{lem:Mkslope}
Suppose $g$ is odd, and let $1 \leq t \leq (a-1)$. Then the first Newton slope of $\prescript{t}{}{X_{(a,g-a)}}$ is at most $\frac{t}{2a}$. In particular $\prescript{a-1}{}{X_{(a,g-a)}}$ is not supersingular and so $a \leq \widetilde{b}_{X_{(a,g-a)}}$.
\end{lemma}

\begin{proof}

By  \cite[Lemma 6.12]{zink68cartier}, the first slope of $\prescript{t}{}{X_{(a,g-a)}}$ is equal to
$$ \lim_{n \to \infty} \tfrac{1}{n}\max\{\ell \in \zz: F_t^n(\prescript{t}{}{M_{(a,g-a)}} ) \subseteq p^\ell ( \prescript{t}{}{M_{(a,g-a)}})\}.$$
Set $s_n = \tfrac{1}{n}\max\{\ell \in \zz: F_t^n( \prescript{t}{}{M_{(a,g-a)}} ) \subseteq p^\ell ( \prescript{t}{}{M_{(a,g-a)}})\}.$

Fix $m \geq 1$. Since even powers of $F_t$ take $M_0$ to $M_0$, there are scalars $w_i \in W(\kk)$ such that
$$F_t^{2am}(e_1) = \sum_{i=1}^g w_ie_i.$$ 
We seek to compute $\mathrm{ord}_p(w_{a+1})$. For any operator $S$ on $M_0$, let $S(e_1)\vert_{e_{a+1}}$ denote the $(a+1)$-th coordinate of $S(e_1)$ under the ordered basis $e_1,\ldots,e_g$. In particular, $w_{a + 1} =  F_t^{2am}(e_1)|_{e_{a+1}}$. 

Recall the operator $\zeta_t = F_t - F$ from \cref{ex:Mk} which is defined as
\begin{align*}
\zeta_t(f_a) &= p^t e_1, & \zeta_t(e_{r+1}) &= -p^{t+1} f_{r+a+1}, & \zeta_t(e_i) &= \zeta_t(f_j) = 0 \text{ for all $i \neq r + 1$, $j \neq a$}
\end{align*}

To understand these operators, the reader might find the diagram in \cref{fig:Fkdiag} below helpful. This diagram records the action of $F$ and $\zeta_t$ on $\{e_1,\ldots,e_{g},f_1,\ldots,f_{g}\}$. The solid arrows give the action of $F$; for example, the arrow \begin{tikzcd}
e_{a+1} \arrow[r, "p"] &[-0.75em] f_{a+1}
\end{tikzcd} records the fact that $F(e_{a+1}) = p f_{a+1}$. The dashed arrows give the action of $\zeta_t$. All but two basis vectors have no dashed arrow emanating from them, reflecting the fact that most basis vectors are in the kernel of $\zeta_t$.

{\vspace{-0.5\baselineskip}
\begin{center}
\begin{figure}[h!]
\begin{adjustbox}{max width=.5\textwidth}
\begin{tikzpicture}

  \fill(-5.76,-2.27) circle (2pt);
  \node at (-5.76,-2.27) (e1) {};
  \node[anchor=east] at (-5.76,-2.27) {$e_1$};

 \fill(-6.55,-1.42)circle (2pt); 
 \node at (-6.55,-1.42) (f1) {}; 
 \node[anchor=east] at (-6.55,-1.42) {$f_{1}$};

 \fill(-6.95,-0.48) circle (2pt); 
 \node at (-6.95,-0.48) (e2) {}; 
 \node[anchor=east] at (-6.95,-0.48) {$e_{2}$};

 \fill(-6.95,0.48) circle (2pt);
 \node at (-6.95,0.48) (f2) {};
 \node[anchor=east] at (-6.95,0.48) {$f_{2}$};

 \fill(-6.55,1.42) circle (2pt); 
 \node at (-6.55,1.42) (ea) {}; 
 \node[anchor=south east] at (-6.55,1.42) {$e_{a}$};

 \fill(-5.76,2.27) circle (2pt); 
 \node at (-5.76,2.27) (fa) {}; 
 \node[anchor=south east] at (-5.76,2.27) {$f_{a}$};

 \fill(-4.64,2.99) circle (2pt);
 \node at (-4.64,2.99) (ea1) {};
 \node[anchor=south east] at (-4.64,2.99) {$e_{a+1}$};

 \fill(-3.25,3.54) circle (2pt); 
 \node at (-3.25,3.54) (fa1) {}; 
 \node[anchor=south east] at (-3.25,3.54) {$f_{a+1}$};

 \fill(-1.67,3.88) circle (2pt); 
 \node at (-1.67,3.88) (ea2) {}; 
 \node[anchor=south] at (-1.67,3.88) {$e_{a+2}$};

 \fill(0,4) circle (2pt); 
  \node at (0,4) (fa2) {}; 
 \node[anchor=south] at (0,4) {$f_{a+2}$};

 \fill(1.67,3.88) circle (2pt); 
  \node at (1.67,3.88) (ea3) {}; 
 \node[anchor=south] at (1.67,3.88) {$e_{a+3}$};

 \fill(3.25,3.54) circle (2pt); 
  \node at (3.25,3.54) (er) {}; 
 \node[anchor=south west] at (3.25,3.54) {$e_{r}$};

 \fill(4.64,2.99) circle (2pt); 
 \node at (4.64,2.99) (fr) {}; 
 \node[anchor=south west] at (4.64,2.99) {$f_{r}$};

 \fill(5.76,2.27) circle (2pt); 
  \node at (5.76,2.27) (er1) {}; 
 \node[anchor=south west] at (5.76,2.27) {$e_{r+1}$};

  \fill(6.55,1.42) circle (2pt);
  \node at (6.55,1.42) (fr1) {};
  \node[anchor=west] at (6.55,1.42) {$f_{r+1}$};

 \fill(6.95,0.48) circle (2pt); 
 \node at (6.95,0.48) (er2) {}; 
 \node[anchor=west] at (6.95,0.48) {$e_{r+2}$};

 \fill(6.95,-0.48) circle (2pt); 
 \node at (6.95,-0.48) (fr2) {}; 
 \node[anchor=west] at (6.95,-0.48) {$f_{r+2}$};

 \fill(6.55,-1.42)  circle (2pt); 
 \node at (6.55,-1.42)  (era1) {}; 
 \node[anchor=west] at (6.55,-1.42)  {$e_{r+a+1}$};

 \fill(5.76,-2.27) circle (2pt); 
 \node at (5.76,-2.27) (fra1) {}; 
 \node[anchor=north west] at (5.76,-2.27) {$f_{r+a+1}$};

 \fill(4.64,-2.99) circle (2pt); 
 \node at (4.64,-2.99) (era2) {}; 
 \node[anchor=north west] at (4.64,-2.99) {$e_{r+a+2}$};

 \fill(3.25,-3.54) circle (2pt);
 \node at (3.25,-3.54) (fra2) {}; 
 \node[anchor=north west] at (3.25,-3.54) {$f_{r+a+2}$};

 \fill(1.67,-3.88) circle (2pt);
 \node at (1.67,-3.88) (era3) {};
 \node[anchor=north west] at (1.67,-3.88) {$e_{r+a+3}$};

 \fill(0,-4) circle (2pt); 
 \node at (0,-4) (fra3) {}; 
 \node[anchor=north] at (0,-4) {$f_{r+a+3}$};

 \fill(-1.67,-3.88) circle (2pt); 
 \node at (-1.67,-3.88) (era4) {}; 
 \node[anchor=north] at (-1.67,-3.88) {$e_{r+a+4}$};

 \fill(-3.25,-3.54) circle (2pt); 
  \node at (-3.25,-3.54) (eg) {};
 \node[anchor=north]  at (-3.25,-3.54) {$e_{g}$};

 \fill(-4.64,-2.99) circle (2pt); 
 \node at (-4.64,-2.99) (fg) {};
 \node[anchor=north] at (-4.64,-2.99) {$f_{g}$};

 \path[->]
        (eg) edge[bend left=5]  node[anchor= west] {}   node[anchor= south west] {$p$} (fg);

 \path[->]
        (fg) edge[bend left=5]  node[anchor= west] {}  (e1);

 \path[->]
        (e1) edge[bend left=10]  node[anchor= west] {}  (f1);

 \path[->]
        (f1) edge[bend left=10]  node[anchor= north west] {}  (e2);

 \path[->]
        (e2) edge[bend left=10]  node[anchor=north  west] {}  (f2);

 \path[-]
        (f2) edge[dotted, bend left=10]   (ea);

 \path[->]
        (ea) edge[bend left=10]  node[anchor= north] {}  (fa);

 \path[->]
        (fa) edge[bend left=5]  node[anchor= north] {}  (ea1);

 \path[->]
        (ea1) edge[bend left=5]  node[anchor= north] {}  node[anchor= north] {$p$}  (fa1);

 \path[->]
        (fa1) edge[bend left=5]  node[anchor= north] {}  (ea2);

 \path[->]
        (ea2) edge[bend left=2]  node[anchor= north] {} node[anchor= north] {$p$}  (fa2);

 \path[->]
        (fa2) edge[bend left=2]  node[anchor= north] {}  (ea3);

 \path[-]
        (ea3) edge[dotted, bend left=5]   (er);

 \path[->]
        (er) edge[bend left=5]  node[anchor= north] {} node[anchor= north east] {$p$}  (fr);

 \path[->]
        (fr) edge[bend left=5]  node[anchor= north east] {}  (er1);

 \path[->]
        (er1) edge[bend left=10]  node[anchor= east] {} node[anchor= north east] {$p$} (fr1);

 \path[->]
        (fr1) edge[bend left=10]  node[anchor=south east] {}  node[anchor=east] {$p$} (er2);
        
 \path[->]
        (er2) edge[bend left=10]  node[anchor=south east] {}  node[anchor=east] {$p$} (fr2);

 \path[-]
        (fr2) edge[dotted, bend left=10]   (era1);

 \path[->]
        (era1) edge[bend left=10]  node[anchor=south] {}   node[anchor=south east] {$p$}  (fra1);

 \path[->]
        (fra1) edge[bend left=5]  node[anchor=south] {}  (era2);

 \path[->]
        (era2) edge[bend left=5]  node[anchor=south] {} node[anchor=south] {$p$} (fra2);

 \path[->]
        (fra2) edge[bend left=5]  node[anchor=south] {}  (era3);

 \path[->]
      (era3) edge[bend left=2]  node[anchor=south] {$p$}     (fra3);

\path[->]
      (fra3) edge[bend left=2]  node[anchor=south] {}     (era4);

 \path[-]
        (era4) edge[dotted, bend left=5]   (eg);

 \path[->]
        (fa) edge[thick, dashed, bend left=30]  node[anchor=south] {}  node[anchor=west] {$p^t$} (e1);

 \path[->]
        (er1) edge[thick, dashed, bend right=30]  node[anchor=east] {$-p^{t+1}$}  node[anchor=north west] {} (fra1);

\end{tikzpicture}
\end{adjustbox}
\setlength{\belowcaptionskip}{-2\baselineskip}
\caption{$F$ and $\zeta_t$ Operators}
  \label{fig:Fkdiag}
\end{figure} 
\end{center}}

The operators $F$ and $\zeta_t$ do not commute, but we can still expand
\begin{align*}
F_t^{2am}(e_1) = (F+\zeta_t)^{2am}(e_1) = \!\!\!\!\!\sum_{ W \in \mathcal{W}_{2am} } \!\!\!\!\!W(e_1), 
\end{align*}
where $\mathcal{W}_{2am}$ is the set of words of length $2am$ in the symbols $F$ and $\zeta_t$. However, for most words $W \in \mathcal{W}_{2am}$, we have that $W(e_1)|_{e_{a+1}}$ is zero. 

It follows immediately from the definitions of $F$ and $\zeta_t$ that $W(e_1)|_{e_{a+1}}$ is nonzero only when $W = F^{2a} W'$ for some word $W'$ of length $2a(m-1)$ in the symbols $F$ and $\zeta_t$, and where $W'(e_1) = w' e_1$ for a nonzero scalar $w'$. 

Further, any such $W'$ must be a composition of the three operators
\begin{align*}
   W_1 &= \zeta_tF^{2a-1}, & W_2 &= F^{g-2a}\zeta_tF^{2r}, \text{ and} & W_3 &= F^{2g}.
\end{align*}

The length of $W_1$ is $2a$ and $W_1(e_1) = p^t e_1$, while $W_2$ is of length $2g-2a$ and $W_2(e_1) = -p^{t+g-2a}e_1$, and $W_3$ is of length $2g$ and $W_3(e_1) = p^g e_1$. Since 
$$\tfrac{t}{2a} < \tfrac{t + g -2a}{2g-2a} < \tfrac{g}{2g},$$
the smallest $p$-adic valuation of $w'$ is achieved when $W' = W_1^{m-1}$, and any other valid $W'$ produces $w'$ with strictly larger $p$-adic valuation.

Therefore,
\begin{align*}
F_t^{2am}(e_1)|_{e_{a+1}} &= F^{2a}(W_1)^{m-1}(e_1) + \text{higher-valuation terms}    \\
&= p^{t(m-1)} e_{a+1} + \text{higher-valuation terms}.
\end{align*}

Then, for each $m \geq 1$  we have $s_{2am} \leq \frac{t(m-1)}{2am},$ and so $\lim_{n \rightarrow \infty} s_n \leq \frac{t}{2a}.$ Thus the first slope of $\prescript{t}{}{X_{(a,g-a)}}$ is at most $\frac{t}{2a}$, and in particular is not $\frac{1}{2}$.

In particular, for $t = a - 1$, $\prescript{a-1}{}{X_{(a,g-a)}}$ is not supersingular and by construction, we have
$$X_{(a,g-a)}[p^{a-1}]  \cong  \prescript{a-1}{}{X_{(a,g-a)}}[p^{a-1}]$$ 
(in a way respecting actions and quasi-polarizations on both sides). Thus, $a \leq \widetilde{b}_{X_{(a,g-a)}}$.
\end{proof}

\begin{example}\label{ex:Mk-notuni}
Let $t \geq 1$, and consider $(a,g-a)$ with $g$ odd. In this example, we construct a $p$-divisible group $\prescript{t}{}{\overline{X}_{(a,g-a)}}$ such that $\prescript{t}{}{\overline{X}_{(a,g-a)}}[p^t] \cong X_{(a,g-a)}[p^t]$. The constructed $p$-divisible group $\prescript{t}{}{\overline{X}_{(a,g-a)}}$ is not equipped with unitary structure; in particular, the subscript $(a,g-a)$ only exists to indicate its relation to the unitary $p$-divisible group $X_{(a,g-a)}$ from \cref{ex:oddMab}.

Let $\prescript{t}{}{\overline{M}_{(a,g-a)}}  = (M_{(a,g-a)}, \overline{F}_t, \overline{V}_t)$, so the underlying $W(\kk)$-module of $\prescript{t}{}{M_{(a,g-a)}}$ agrees with that of $M_{(a,g-a)}$ in \cref{ex:oddMab}, except for $\overline{F}_t$ and $\overline{V}_t$, which we obtain by perturbing $F$ and $V$. Define $\overline{F}_t \coloneqq F + \theta_t$ and $\overline{V}_t \coloneqq V + \kappa_t$, where $\theta_t$ and $\kappa_t$ are $0$ on all basis elements except for the following:
\begin{align*}
    \theta_t(e_{a+1}) &= p^t e_1, & \kappa_t(f_{a+1}) &= -p^{t}f_g.
\end{align*}

By construction, $F \equiv \overline{F}_t \mod p^t$ and $V \equiv \overline{V}_t \mod p^t$ and a direct verification shows $\overline{F}_t \circ \overline{V}_t = \overline{V}_t \circ \overline{F}_t = p$. Therefore, for each $t \geq 1$, we've constructed a $p$-divisible group $\prescript{t}{}{\overline{X}_{(a,g-a)}}$, such that $\prescript{t}{}{\overline{X}_{(a,g-a)}}[p^t] \cong X_{(a,g-a)}[p^t]$. 

We do not equip the $p$-divisible group $\prescript{t}{}{\overline{X}_{(a,g-a)}}$ with the structure of a unitary $p$-divisible group, and the naive extension of the structure from \cref{ex:oddMab} fails to meet the necessary conditions. To see this, recall the grading $M_0 = \linspan_{W(\kk) }  \{ e_1, \dots, e_{g} \}$ and $M_1 = \linspan_{W(\kk) } \{  f_1, \dots, f_{g} \}$ of $M_{(a,g-a)}$ from \cref{ex:oddMab}. The operator $\theta_t$ (resp. $\kappa_t$) is such that $\theta_t M_0 \subseteq M_0$ (resp. $\kappa_t M_1 \subseteq M_1$) and therefore $\overline{F}_t$ (resp. $\overline{V}_t$) is not grade-reversing, as is required for a unitary $p$-divisible group.
\end{example}

\begin{lemma}\label{lem:Mkslope-notuni}
Suppose $g$ is odd, and let $1 \leq t \leq a$. Then the first Newton slope of $\prescript{t}{}{\overline{X}_{(a,g-a)}}$ is at most $\frac{t}{2a+1}$. In particular $\prescript{a}{}{\overline{X}_{(a,g-a)}}$ is not supersingular and $a + 1 \leq b_{X_{(a,g-a)}}$. 
\end{lemma}

\begin{proof}

By  \cite[Lemma 6.12]{zink68cartier}, the first slope of $\prescript{t}{}{\overline{X}_{(a,g-a)}}$ is equal to
$$ \lim_{n \to \infty} \tfrac{1}{n}\max\{\ell \in \zz: \overline{F}_t^n(\prescript{t}{}{\overline{M}_{(a,g-a)}} ) \subseteq p^\ell ( \prescript{t}{}{\overline{M}_{(a,g-a)}})\}.$$
Set $s_n = \tfrac{1}{n}\max\{\ell \in \zz: \overline{F}_t^n( \prescript{t}{}{\overline{M}_{(a,g-a)}} ) \subseteq p^\ell ( \prescript{t}{}{\overline{M}_{(a,g-a)}})\}.$

For any operator $\overline{S}$ on $M_{(a,g-a)}$, let $\overline{S}(e_1)\vert_{e_{1}}$ denote the $e_1$-coordinate of $\overline{S}(e_1)$ under the ordered basis $e_1,\ldots,e_g,f_1,\ldots,f_g$. Fixing $m \geq 1$, we seek to compute $\mathrm{ord}_p(\overline{F}_t^{(2a+1)m}(e_1)|_{e_{1}})$. 

Recall the operator $\theta_t = \overline{F}_t - F$ from \cref{ex:Mk-notuni} which is defined as
\begin{align*}
\theta_t(e_{a+1}) &= p^t e_1, & \theta_t(e_i) &= \theta_t(f_j) = 0 \text{ for all $i \neq a + 1$ and all $j$}
\end{align*}

Similar to before, the diagram in \cref{fig:bar-Fkdiag} records the action of $F$ and $\theta_t$ on $\{e_1,\ldots,e_{g},f_1,\ldots,f_{g}\}$. The solid arrows give the action of $F$ and the dashed arrow give the action of $\theta_t$. All but one basis vectors have no dashed arrows emanating from them, reflecting the fact that most basis vectors are in the kernel of $\theta_t$.

\vspace*{-\baselineskip}
{
\begin{center}
\begin{figure}[h!]
\begin{adjustbox}{max width=.5\textwidth}
\begin{tikzpicture}

  \fill(-5.76,-2.27) circle (2pt);
  \node at (-5.76,-2.27) (e1) {};
  \node[anchor=east] at (-5.76,-2.27) {$e_1$};

 \fill(-6.55,-1.42)circle (2pt); 
 \node at (-6.55,-1.42) (f1) {}; 
 \node[anchor=east] at (-6.55,-1.42) {$f_{1}$};

 \fill(-6.95,-0.48) circle (2pt); 
 \node at (-6.95,-0.48) (e2) {}; 
 \node[anchor=east] at (-6.95,-0.48) {$e_{2}$};

 \fill(-6.95,0.48) circle (2pt);
 \node at (-6.95,0.48) (f2) {};
 \node[anchor=east] at (-6.95,0.48) {$f_{2}$};

 \fill(-6.55,1.42) circle (2pt); 
 \node at (-6.55,1.42) (ea) {}; 
 \node[anchor=south east] at (-6.55,1.42) {$e_{a}$};

 \fill(-5.76,2.27) circle (2pt); 
 \node at (-5.76,2.27) (fa) {}; 
 \node[anchor=south east] at (-5.76,2.27) {$f_{a}$};

 \fill(-4.64,2.99) circle (2pt);
 \node at (-4.64,2.99) (ea1) {};
 \node[anchor=south east] at (-4.64,2.99) {$e_{a+1}$};

 \fill(-3.25,3.54) circle (2pt); 
 \node at (-3.25,3.54) (fa1) {}; 
 \node[anchor=south east] at (-3.25,3.54) {$f_{a+1}$};

 \fill(-1.67,3.88) circle (2pt); 
 \node at (-1.67,3.88) (ea2) {}; 
 \node[anchor=south] at (-1.67,3.88) {$e_{a+2}$};

 \fill(0,4) circle (2pt); 
  \node at (0,4) (fa2) {}; 
 \node[anchor=south] at (0,4) {$f_{a+2}$};

 \fill(1.67,3.88) circle (2pt); 
  \node at (1.67,3.88) (ea3) {}; 
 \node[anchor=south] at (1.67,3.88) {$e_{a+3}$};

 \fill(3.25,3.54) circle (2pt); 
  \node at (3.25,3.54) (er) {}; 
 \node[anchor=south west] at (3.25,3.54) {$e_{r}$};

 \fill(4.64,2.99) circle (2pt); 
 \node at (4.64,2.99) (fr) {}; 
 \node[anchor=south west] at (4.64,2.99) {$f_{r}$};

 \fill(5.76,2.27) circle (2pt); 
  \node at (5.76,2.27) (er1) {}; 
 \node[anchor=south west] at (5.76,2.27) {$e_{r+1}$};

  \fill(6.55,1.42) circle (2pt);
  \node at (6.55,1.42) (fr1) {};
  \node[anchor=west] at (6.55,1.42) {$f_{r+1}$};

 \fill(6.95,0.48) circle (2pt); 
 \node at (6.95,0.48) (er2) {}; 
 \node[anchor=west] at (6.95,0.48) {$e_{r+2}$};

 \fill(6.95,-0.48) circle (2pt); 
 \node at (6.95,-0.48) (fr2) {}; 
 \node[anchor=west] at (6.95,-0.48) {$f_{r+2}$};

 \fill(6.55,-1.42)  circle (2pt); 
 \node at (6.55,-1.42)  (era1) {}; 
 \node[anchor=west] at (6.55,-1.42)  {$e_{r+a+1}$};

 \fill(5.76,-2.27) circle (2pt); 
 \node at (5.76,-2.27) (fra1) {}; 
 \node[anchor=north west] at (5.76,-2.27) {$f_{r+a+1}$};

 \fill(4.64,-2.99) circle (2pt); 
 \node at (4.64,-2.99) (era2) {}; 
 \node[anchor=north west] at (4.64,-2.99) {$e_{r+a+2}$};

 \fill(3.25,-3.54) circle (2pt);
 \node at (3.25,-3.54) (fra2) {}; 
 \node[anchor=north west] at (3.25,-3.54) {$f_{r+a+2}$};

 \fill(1.67,-3.88) circle (2pt);
 \node at (1.67,-3.88) (era3) {};
 \node[anchor=north west] at (1.67,-3.88) {$e_{r+a+3}$};

 \fill(0,-4) circle (2pt); 
 \node at (0,-4) (fra3) {}; 
 \node[anchor=north] at (0,-4) {$f_{r+a+3}$};

 \fill(-1.67,-3.88) circle (2pt); 
 \node at (-1.67,-3.88) (era4) {}; 
 \node[anchor=north] at (-1.67,-3.88) {$e_{r+a+4}$};

 \fill(-3.25,-3.54) circle (2pt); 
  \node at (-3.25,-3.54) (eg) {};
 \node[anchor=north]  at (-3.25,-3.54) {$e_{g}$};

 \fill(-4.64,-2.99) circle (2pt); 
 \node at (-4.64,-2.99) (fg) {};
 \node[anchor=north] at (-4.64,-2.99) {$f_{g}$};

 \path[->]
        (eg) edge[bend left=5]  node[anchor= west] {}   node[anchor= south west] {$p$} (fg);

 \path[->]
        (fg) edge[bend left=5]  node[anchor= west] {}  (e1);

 \path[->]
        (e1) edge[bend left=10]  node[anchor= west] {}  (f1);

 \path[->]
        (f1) edge[bend left=10]  node[anchor= north west] {}  (e2);

 \path[->]
        (e2) edge[bend left=10]  node[anchor=north  west] {}  (f2);

 \path[-]
        (f2) edge[dotted, bend left=10]   (ea);

 \path[->]
        (ea) edge[bend left=10]  node[anchor= north] {}  (fa);

 \path[->]
        (fa) edge[bend left=5]  node[anchor= north] {}  (ea1);

 \path[->]
        (ea1) edge[bend left=5]  node[anchor= north] {}  node[anchor= north] {$p$}  (fa1);

 \path[->]
        (fa1) edge[bend left=5]  node[anchor= north] {}  (ea2);

 \path[->]
        (ea2) edge[bend left=2]  node[anchor= north] {} node[anchor= north] {$p$}  (fa2);

 \path[->]
        (fa2) edge[bend left=2]  node[anchor= north] {}  (ea3);

 \path[-]
        (ea3) edge[dotted, bend left=5]   (er);

 \path[->]
        (er) edge[bend left=5]  node[anchor= north] {} node[anchor= north east] {$p$}  (fr);

 \path[->]
        (fr) edge[bend left=5]  node[anchor= north east] {}  (er1);

 \path[->]
        (er1) edge[bend left=10]  node[anchor= east] {} node[anchor= north east] {$p$} (fr1);

 \path[->]
        (fr1) edge[bend left=10]  node[anchor=south east] {}  node[anchor=east] {$p$} (er2);
        
 \path[->]
        (er2) edge[bend left=10]  node[anchor=south east] {}  node[anchor=east] {$p$} (fr2);

 \path[-]
        (fr2) edge[dotted, bend left=10]   (era1);

 \path[->]
        (era1) edge[bend left=10]  node[anchor=south] {}   node[anchor=south east] {$p$}  (fra1);

 \path[->]
        (fra1) edge[bend left=5]  node[anchor=south] {}  (era2);

 \path[->]
        (era2) edge[bend left=5]  node[anchor=south] {} node[anchor=south] {$p$} (fra2);

 \path[->]
        (fra2) edge[bend left=5]  node[anchor=south] {}  (era3);

 \path[->]
      (era3) edge[bend left=2]  node[anchor=south] {$p$}     (fra3);

\path[->]
      (fra3) edge[bend left=2]  node[anchor=south] {}     (era4);

 \path[-]
        (era4) edge[dotted, bend left=5]   (eg);

 \path[->]
        (ea1) edge[thick, dashed, bend left=30]  node[anchor=south] {}  node[anchor=west] {$p^t$} (e1);

\end{tikzpicture}
\end{adjustbox}
\setlength{\belowcaptionskip}{-2\baselineskip}
\caption{$F$ and $\theta_t$ Operators}
  \label{fig:bar-Fkdiag}
\end{figure} 
\end{center}}

The operators $F$ and $\theta_t$ do not commute, but we can still expand
\begin{align*}
\overline{F}_t^{(2a+1)m}(e_1) = (F+\theta_t)^{(2a+1)m}(e_1) = \!\!\!\!\!\sum_{ W \in \mathcal{W}_{(2a+1)m} } \!\!\!\!\!W(e_1), 
\end{align*}
where $\mathcal{W}_{(2a+1)m}$ is the set of words of length $(2a+1)m$ in the symbols $F$ and $\theta_t$. However, for most words $W \in \mathcal{W}_{(2a+1)m}$, we have that $W(e_1)|_{e_{1}}$ is zero. 

It follows immediately from the definitions of $F$ and $\theta_t$ that $W(e_1)|_{e_{1}}$ is nonzero only when $W$ is a composition of the two operators
\begin{align*}
   W_1 &= \theta_tF^{2a}, \text{ and} &
   W_2 &= F^{2g}.
\end{align*}

The length of $W_1$ is $2a + 1$ and $W_1(e_1) = p^t e_1$, and $W_2$ is of length $2g$ and $W_2(e_1) = p^g e_1$. Since $t \leq a$, by assumption, we have
$$\tfrac{t}{2a+1} < \tfrac{g}{2g},$$
and hence the smallest $p$-adic valuation of $W(e_1)\vert_{e_1}$ is achieved when $W = W_1^{m}$, and any other valid $W$ produces $W(e_1)\vert_{e_1}$ with strictly larger $p$-adic valuation.

Therefore,
\begin{align*}
\overline{F}_t^{(2a+1)m}(e_1)|_{e_{1}} &= W_1^{m}(e_1) + \text{higher-valuation terms}    \\
&= p^{tm} e_{1} + \text{higher-valuation terms}.
\end{align*}

Then, for each $m \geq 1$  we have $s_{(2a+1)m} \leq \frac{tm}{(2a+1)m} = \frac{t}{2a+1},$ and so $\lim_{n \rightarrow \infty} s_n \leq \frac{t}{2a+1}.$ Thus the first slope of $\prescript{t}{}{X_{(a,g-a)}}$ is at most $\frac{t}{2a+1} \leq \frac{a}{2a + 1} < \frac{1}{2}$, and in particular is not $\frac{1}{2}$.

In particular, for $t = a$, $\prescript{a}{}{\overline{X}_{(a,g-a)}}$ is not supersingular and by construction, we have
$$X_{(a,g-a)}[p^{a}]  \cong  \prescript{a}{}{\overline{X}_{(a,g-a)}}[p^{a}].$$ 
Thus, $a+1 \leq b_{X_{(a,g-a)}}$.
\end{proof}

We use \cref{ex:Mk,ex:Mk-notuni}, \cref{lem:Mkslope,lem:Mkslope-notuni} and the results on minimal height (\cref{thm:allabtminht}) to understand both isogeny cutoffs and the supersingular isogeny bounds.

Recall, from \cref{thm:allabtminht}, the number $q_{\max} = \begin{cases}
    a & \text{if $g > 2a$;}\\
    a - 1 &\text{if $g = 2a$.}
\end{cases}$

\begin{theorem}\label{thm:isogenycutoffs}

Let $X$ be any supersingular unitary $p$-divisible group of signature $(a,g-a)$ over $\kk$. Then,
\[b_X \leq 1 + q_{\max} = \begin{cases}
    a + 1, & \text{if $g>2a$;}\\
    a, & \text{if $g = 2a$.}
\end{cases}\]

Furthermore, we have the following constraints on the supersingular isogeny bounds
\[\max\{1,q_{\max}\} \leq \widetilde{B}_{(a,g-a)}  \leq B_{(a,g-a)} = 1 + q_{\max},\]
equivalently,
\[\begin{cases}
    a \leq \widetilde{B}_{(a,g-a)} \leq B_{(a,g-a)} = a+1, & \text{when $g > 2a$;}\\[0.75em]
    \max\{1,a-1\} \leq \widetilde{B}_{(a,a)}  \leq B_{(a,a)} = a, & \text{when $g = 2a$.}
\end{cases}\]
\end{theorem}

\begin{proof}

We showed in \cref{thm:allabtminht} that, for any supersingular unitary $p$-divisible group of $X$ of signature $(a,g-a)$, its minimal height $q_X$ is such that $q_X \leq q_{\max}$. \cite[Lemma~6.3]{TraversoII} gives us $b_X \leq 1 + q_X$  (cf. \cite[Theorem~1.2 \& Remark~6.4]{TraversoII}) and thus
\[b_X \leq 1 + q_{\max} = \begin{cases}
    a+1, & \text{if $g>2a$;}\\
    a, & \text{if $g = 2a$.}
\end{cases}\tag{$\triangle$}\label{eq:isogeny-bound}\]

We first prove that claim that $B_{(a,g-a)} = 1 + q_{\max}$ by presenting supersingular unitary $p$-divisible groups $Y$ of signature $(a,g-a)$ such that $b_Y = 1 + q_{\max}$. This candidate is built using the constructions in \cref{ex:oddMab} and \cref{ex:Mk-notuni}.

First consider the case when $g > 2a$ and $g$ is odd, and recall the $p$-divisible group $Y = X_{(a,g-a)}$. We have shown in \cref{lem:Mkslope-notuni} that $b_{Y} \geq a + 1$. Combined with the general upper bound on the isogeny cutoff \cref{eq:isogeny-bound}, we conclude $b_{Y} = a + 1$.

Assume now $g > 2a$ and $g$ is even. Since $g > 2a - 1$, consider $Y = X_{(a,g-a)} = X_{(a,g-a-1)} \times_\kk X_{(0,1)}$ of \cref{ex:evenMab} and let $\overline{Y} = \prescript{a}{}{\overline{X}_{(a,g-a-1)}}   \times_\kk X_{(0,1)}$. As in the previous case, $Y$ is supersingular while $\overline{Y}$ is not supersingular, and $Y[p^{a}] \cong \overline{Y}[p^{a}]$; therefore $a + 1 \leq b_Y$. Combined with the general upper bound on the isogeny cutoff \cref{eq:isogeny-bound}, we conclude $b_{Y} = a + 1$.

When $g = 2a$, we consider $Y = X_{(a,a)} = X_{(a-1,a)} \times_\kk X_{(1,0)}$ of \cref{ex:Maa} and $\overline{Y} =   \prescript{a-1}{}{X_{(a-1,a)}}    \times_\kk X_{(1,0)}$. Since $Y$ is supersingular, while $\overline{Y}$ is not supersingular, and $Y[p^{a-1}] \cong \overline{Y}[p^{a-1}],$ we have $a \leq b_Y$. Combined with the general upper bound on the isogeny cutoff \cref{eq:isogeny-bound}, we conclude $b_{Y} = a$.

Finally, from the definition $B_{(a,g-a)}$, we obtain $B_{(a,g-a)} = 1 + q_{\max}$.

We now focus on the claim regarding $\widetilde{B}_{(a,g-a)}$. Since we already have $1 \leq \widetilde{B}_{(a,g-a)} \leq B_{(a,g-a)}$, we therefore focus our efforts on the claim $q_{\max} \leq \widetilde{B}_{(a,g-a)}$. For this we rely on constructions in \cref{ex:oddMab} and \cref{ex:Mk}.

First consider the case when $g > 2a$ and $g$ is odd, and recall the $p$-divisible group $Y = X_{(a,g-a)}$ constructed in \cref{ex:oddMab}. We have shown in \cref{lem:Mkslope} that $a \leq \widetilde{b}_{Y}$ and therefore $a \leq \widetilde{B}_{(a,g-a)}$. 

Assume now $g > 2a$ and $g$ is even. Since $g > 2a - 1$, consider $Y = X_{(a,g-a)} = X_{(a,g-a-1)} \times_\kk X_{(0,1)}$ of \cref{ex:evenMab} and $Y' \coloneqq \prescript{a-1}{}{X_{(a,g-a-1)}} \times_\kk X_{(0,1)}$. As in the previous case, $Y$ is supersingular while $Y'$ is not supersingular, and $Y[p^{a-1}] \cong Y'[p^{a-1}]$; therefore $a \leq \widetilde{b}_Y$. Hence $a \leq \widetilde{B}_{(a,g-a)}$. 

For $g = 2a$, consider $Y = X_{(a,a)} = X_{(a-1,a)} \times_\kk X_{(1,0)}$ of \cref{ex:Maa} and $Y' \coloneqq  \prescript{a-2}{}{X_{(a-1,a)}}    \times_\kk X_{(1,0)}$. Since $Y$ is supersingular, while $Y'$ is not supersingular, and $Y[p^{a-2}] \cong Y'[p^{a-2}],$ we have $a - 1 \leq \widetilde{b}_Y$. Hence $a-1 \leq \widetilde{B}_{(a,g-a)}$.
\end{proof}

\begin{remark}
There are only a few values of $(a,g-a)$ for which the enhanced supersingular isogeny bound of signature $(a,g-a)$ is known. By a unitary Shimura varity of signature $(a,g-a)$ we mean the characteristic $p$ fiber of a Shimura variety for $\mathsf{GU}(a, g-a)$ where $p > 2$ is inert (in a relevant field) and the level structure is hyperspecial at $p$ (see Definition \ref{def:unitary-shimura} in the next section).

As a degenerate example, for any $g > 0$, the signature $(0,g)$ unitary ``Shimura variety" has a single Newton stratum (which is the supersingular locus). Though in this paper we generally require $a > 0$, if one extended the definitions of $\widetilde{b}_X$ and $\widetilde{B}_{(a,g-a)}$ to include this degenerate case, every point every point would have isogeny cutoff $0$, and so $\widetilde{B}_{(0,g)}=0$.

It is shown in \cite{bultel2006congruence} that unitary Shimura varieties of signature $(1,g-1)$ have the property that the supersingular locus is a union of Ekedahl--Oort strata. As a result, when $g > 1$, we have $\widetilde{B}_{(1,g-1)} = 1$. Similarly, it is shown in \cite{GoertzHe} that unitary Shimura varieties of signature $(2,2)$ also have the property that the supersingular locus is a union of Ekedahl--Oort strata, and again in this case both $\widetilde{B}_{(2,2)} = 1$.

One might be tempted to say, from these examples, that we always have $\widetilde{B}_{(a,g-a)} = a$ when $g > 2a$, and $\widetilde{B}_{(a,a)} = a - 1$ when $g = 2a$ and $a > 1$. However, it is important to note that all of the above examples come from Shimura varieties of \emph{Coxeter type}, i.e., ones with the property that the supersingular locus is a union of Ekedahl-Oort strata, and each Ekedahl-Oort stratum is a union of Deligne-Lusztig sets associated to certain Coxeter elements. However, most unitary Shimura varieties are not of Coxeter type  (cf. \cite[Theorem 5.1.2]{GoertzHe} for a complete classification of this phenomenon), so one should use caution when extrapolating from these facts to the general situation.
\end{remark}

\section{Applications to Abelian Varieties and Their Moduli Spaces}\label{sec:SVandBTn}

In this section we briefly translate \cref{thm:isogenycutoffs} into the language of the $\bt_m$ stratifications of unitary Shimura varieties. We recall the the PEL datum and the moduli interpretation of the $\kk$-points of our unitary Shimura varieties, see for example \cite{Vollaard,VollaardWedhorn,KRGlobal} for more details.

\begin{definition}\label{def:unitary-shimura}
For a prime $p>2$, and $a$ and $g$ as in \cref{not:a-n-g}, fix an imaginary quadratic field $L$ in which the prime $p$ is inert. Consider the PEL datum $(L, \mathcal{O}_L \otimes_{\zz } \zz_{(p)}, {}^*, V, (\cdot,\cdot),\Lambda, \boldsymbol{\mathsf{G}}, h)$ as in \cite{kottwitz1992points}, where
\begin{itemize}
\item{$L$ is the imaginary quadratic field as above,  with ${}^*$ being the nontrivial automorphism of $L$ over $\qq$;}
\item{$V$ is an $L$-vector space of dimension $g$, equipped with a perfect alternating $\qq$-bilinear pairing $(\cdot,\cdot): V \times V \to \qq$ such that $(xv,w) = (v, x^*w)$ for all $x\in L$ and $v,w \in V$; }
\item{$\boldsymbol{\mathsf{G}}$ is the algebraic group of $L$-linear symplectic similitudes of $(V, (\cdot,\cdot) )$. We assume that $\boldsymbol{\mathsf{G}}_{\mathbx{R}}$ is isomorphic to the real algebraic group $\mathsf{GU}(a,g-a)$;}
\item{$\Lambda$ is an $\mathcal{O}_L \otimes_{\zz } \zz_{(p)}$-invariant $\zz_p$-lattice of $V \otimes_{\qq} \qq_p$ such that the alternating form induced by $(\cdot,\cdot)$ is a perfect $\zz_p$-form;}
\item{$h: \mathrm{Res}_{\cc/\rr}(\mathbx{G}_{m,\cc}) \to \boldsymbol{\mathsf{G}}_\rr$ is the homomorphism of real algebraic groups mapping $z \mapsto \mathrm{diag}(z^a, \overline{z}^{g-a})$. }
\end{itemize}

Fix a compact open subgroup $C^p \subset \boldsymbol{\mathsf{G}}(\mathbx{A}_f^p)$, where $\mathbx{A}_f^p$ is the ring of finite adeles with a trivial component at $p$, and define $C_p = \boldsymbol{\mathsf{G}}(\zz_p)$ and $C = C^pC_p$. For $C^p$ small enough, Kottwitz in \cite{kottwitz1992points} attaches to this PEL datum a smooth, quasi-projective scheme $\mathbf{M}(a,g-a)_{C}$ over $\spec(\mathcal{O}_{L,(p)})$ when $g \neq 2a$ and over $\spec(\zz_{(p)})$ when $g=2a$. This is an integral model of a classical complex Shimura variety; the condition that $C_p = \boldsymbol{\mathsf{G}}(\zz_p)$ is the requirement that the level structure is hyperspecial at $p$.

The integral model $\mathbf{M}(a,g-a)_{C}$ has relative dimension  $a(g-a)$. Let $\mathcal{M}(a,g-a)_{C}$ be the fiber at $p$ of $\mathbf{M}(a,g-a)_{C}$; we call this a characteristic $p$ \textbf{unitary Shimura variety}. In particular, $\mathcal{M}(a,g-a)_C$ is defined over the residue field $\ff_{p^2}$ of $L$ at $p$ when $g \neq 2a$ (and over $\ff_p$ when $g=2a$) and is of dimension $a(g-a)$. 

The $\kk$-points of $\mathcal{M}(a,g-a)_C$ correspond to isomorphism classes of tuples $(A, \lambda, \iota, \eta)$, where
 
\begin{itemize}
    \item $A$ is an abelian variety over $\kk$ of dimension $g$;
    \item $\lambda: A \overset{\sim}{\to} A^\vee$ is a prime-to-$p$ polarization; 
    \item $\iota: \mathcal{O}_L \hookrightarrow \mathrm{End}(A) \otimes_{\zz } \zz_{(p)}$ is an action of signature $(a,g-a)$ satisfying the Rosati involution condition; and 
    \item $\eta : H_1(A, \mathbx{A}_f^p) \xrightarrow{\sim} V \otimes_{\qq} \mathbx{A}_f^p \ \mod C^p$ is the  prime-to-$p$ level structure.
\end{itemize}
\end{definition}

The \emph{Newton stratification} of $\mathcal{M}(a,g-a)_C$ is based on the isogeny classes of $p$-divisible groups. In particular, two $\kk$-points $(A_1, \iota_1, \lambda_1, \eta_1)$ and $(A_2, \iota_2, \lambda_2, \eta_2)$ are in the same Newton stratum if and only if $A_1$ and $A_2$ have the same Newton polygon. There is a unique closed Newton stratum, the \emph{supersingular locus} $\mathcal{M}(a,g-a)_C^{\mathrm{ss}}$, which parameterizes \emph{supersingular} abelian varieties with additional structure. 

For the \emph{Ekedahl--Oort stratification} of $\mathcal{M}(a,g-a)_C$,  two $\kk$-points $(A_1, \iota_1, \lambda_1, \eta_1)$ and $(A_2, \iota_2, \lambda_2, \eta_2)$ are in the same Ekedahl--Oort stratum if and only if $A_1[p]$ and $A_2[p]$, equipped with their induced action and quasi-polarization, are isomorphic. 

Generalizing the Ekedahl--Oort stratification is the \emph{$\bt_m$ stratification} of $\mathcal{M}(a,g-a)_C$: for $m =1$, this is exactly the Ekedahl--Oort stratification. In general, the $\bt_m$ stratification (defined for Hodge type Shimura varieties by Vasiu in \cite{vasiu2006level} where the term ``level-$m$ stratification'' is used instead) is defined by the property that two $\kk$-points are in the same $\bt_m$ stratum if and only if the group schemes $A_1[p^m]$ and $A_2[p^m]$ are isomorphic, with their induced action and quasi-polarization. The reader should note that this is a stratification in the sense of \cite[Remark 2.1.1]{Vasiubounded}; in particular, for most $(m,a,g)$, the moduli space $\mathcal{M}(a,g-a)_C$ has infinitely many $\bt_m$ strata.

Recall our notation $$q_{\max} \coloneqq \begin{cases}
    a & \text{if $g > 2a$;}\\
    a - 1 &\text{if $g = 2a$.}
\end{cases}$$ Given a $p^m$-torsion group scheme $G$ with action and quasi-polarization, we'll denote the $\bt_m$ stratum defined by $G$ as $\mathcal{M}(a,g-a)_{C,G}$. The corollary below follows from applying \cref{thm:isogenycutoffs} in the language of stratifications of $\mathcal{M}(a,g-a)_C$.
\begin{corollary}\label{cor:containedordisjoint}
Let $m\geq  q_{\max} +1 $ be an integer. 

For each $\bt_m$ stratum $\mathcal{M}(a,g-a)_{C,G}$, and for any algebraically closed field $\kk$ containing $\overline{\ff}_p$, either
$$\mathcal{M}(a,g-a)_{C,G} (\kk) \subseteq \mathcal{M}(a,g-a)_C^{\mathrm{ss}}(\kk)$$
or 
$$\mathcal{M}(a,g-a)_{C,G} (\kk) \cap \mathcal{M}(a,g-a)_C^{\mathrm{ss}}(\kk) = \emptyset.$$
\end{corollary}

\begin{proof}
We show that, when $\mathcal{M}(a,g-a)_{C,G} (\kk) \cap \mathcal{M}(a,g-a)_{C,G}^{\mathrm{ss}}(\kk)$ is nonempty, we have the containment $\mathcal{M}(a,g-a)_{C,G} (\kk) \subseteq \mathcal{M}(a,g-a)_C^{\mathrm{ss}}(\kk)$.

Consider a $\kk$-point $(A_1, \iota_1, \lambda_1, \eta_1)$ in $\mathcal{M}(a,g-a)_{C,G} (\kk) \cap \mathcal{M}(a,g-a)_C^{\mathrm{ss}}(\kk)$, and let $(A_2, \iota_2, \lambda_2, \eta_2)$ be any other point of $\mathcal{M}(a,g-a)_{C,G}$. Let $(X_1, \iota_1, \lambda_1)$ and $(X_2, \iota_2, \lambda_2)$ be the unitary $p$-divisible groups of signature $(a,g-a)$ defined from these points. 

By the assumption that these two points both lie in the same $\bt_m$ stratum $\mathcal{M}(a,g-a)_{C,G}$, we have that $X_1[p^m]$ and $X_2[p^m]$ are isomorphic (with their induced action and quasi-polarization). \cref{thm:isogenycutoffs} tells us that $\widetilde{B}_{(a,g-a)} \leq m$. So, in particular, we have $\widetilde{b}_{X_1} \leq m$. 

Since $X_1[p^m] \cong X_2[p^m]$ and $\widetilde{b}_{X_1} \leq m$, the $p$-divisible groups $X_1$ and $X_2$ are isogenous, and so the point $(A_2, \iota_2, \lambda_2, \eta_2)$ is also contained in the supersingular locus $\mathcal{M}(a,g-a)_C^{\mathrm{ss}}$.
\end{proof}

\begin{corollary}\label{cor:union}
Let $m \geq q_{\max} + 1$ be an integer. The $\bt_m$ stratification of $\mathcal{M}(a,g-a)_C$ restricts to a stratification of the supersingular locus in the following sense: for every algebraically closed field $\kk$ over $\overline{\ff}_p$ there is a set $\mathcal{G}(m,\kk)$ of $p^m$-torsion group schemes $G$ (equipped with action and quasi-polarization), each with corresponding $\bt_m$ stratum $\mathcal{M}(a,g-a)_{C,G}$, such that
$$\bigsqcup_{G \in \mathcal{G}(m,\kk) } \mathcal{M}(a,g-a)_{C,G}(\kk ) = \mathcal{M}(a,g-a)_C^{\mathrm{ss}}(\kk ).$$
Further, for each $G \in \mathcal{G}(m,\kk)$, the morphism $\mathcal{M}(a,g-a)_{C,G} \rightarrow (\mathcal{M}(a,g-a)_C)_\kk$ factors through the supersingular locus $(\mathcal{M}(a,g-a)_C^{\mathrm{ss}})_{\kk}$.
\end{corollary}

\begin{proof}
Let $\kk$ be an algebraically closed field containing $\overline{\ff}_p$. Let 
$$\mathcal{G}(m,\kk) = \{ p^m\text{-torsion} \ G \ | \ \mathcal{M}(a,g-a)_{C,G}(\kk ) \cap \mathcal{M}(a,g-a)_C^{\mathrm{ss}}(\kk ) \neq \emptyset \},$$
with the understanding that $G$ above denotes a representative of an isomorphism class of $p^m$-torsion group schemes over $\kk$ with action and quasi-polarization, isomorphisms being required to respect the additional structure, defining a $\bt_m$ stratum of $(\mathcal{M}(a,g-a)_C)_\kk$. 

Then, the fact that 
$$\bigsqcup_{G \in \mathcal{G}(m,\kk) } \mathcal{M}(a,g-a)_{C,G}(\kk ) = \mathcal{M}(a,g-a)_C^{\mathrm{ss}}(\kk )$$
follows directly from \cref{cor:containedordisjoint}.

Also by \cref{cor:containedordisjoint}, for any $G \in \mathcal{G}(m,\kk)$ and for any algebraically closed field $\kappa$ containing $\kk$, we have an inclusion of $\kappa$-points $\mathcal{M}(a,g-a)_{C,G}(\kappa) \subseteq \mathcal{M}(a,g-a)_C^{\mathrm{ss}}(\kappa)$. Therefore, the image of the map $|\mathcal{M}(a,g-a)_{C,G}| \rightarrow  |(\mathcal{M}(a,g-a)_C)_\kk|$ is contained in $|(\mathcal{M}(a,g-a)_C^{\mathrm{ss}})_{\kk}|$.

Now, since the $\bt_m$ stratum $\mathcal{M}(a,g-a)_{C,G}$ is considered as a subscheme of $(\mathcal{M}(a,g-a)_C)_\kk$ with its induced reduced structure (as in \cite{vasiu2006level}) and $(\mathcal{M}(a,g-a)_C^{ss})_\kk$ is a closed subvariety of $(\mathcal{M}(a,g-a)_C)_\kk$, the morphism $\mathcal{M}(a,g-a)_{C,G} \rightarrow (\mathcal{M}(a,g-a)_C)_\kk$ factors through the supersingular locus $(\mathcal{M}(a,g-a)_C^{\mathrm{ss}})_{\kk}$.
\end{proof}

\renewcommand\bibpreamble{\vspace*{-0.1\baselineskip}}

\bibliographystyle{alpha}
\bibliography{Bibliography.bib}

\end{document}